\documentclass[11pt]{article}

\usepackage[centertags]{amsmath}
\usepackage{amsfonts}
\usepackage{amssymb}
\usepackage{amsthm}
\usepackage{newlfont}
\usepackage{bibspacing}
\usepackage{verbatim}

\newtheorem{thm}{Theorem}[section]
\newtheorem*{thm*}{Theorem}
\newtheorem{cor}[thm]{Corollary}

\newtheorem{lem}[thm]{Lemma}

\newtheorem{prop}[thm]{Proposition}
\newtheorem*{prop*}{Proposition}

\newtheorem{ques}{Question}
\newtheorem{conj}{Conjecture}
\newtheorem*{conj1}{Conjecture \textbf{$1^*$}}
\newtheorem*{conj2}{Conjecture \textbf{$2^*$}}
\newtheorem*{conj12}{Conjectures \textbf{1} and \textbf{2}}

\newcommand{\conjone}{$\mathop{1^*\;}$}
\newcommand{\conjtwo}{$\mathop{2^*\;}$}
\newtheorem*{conj*}{Conjecture}

\newtheorem*{dfn*}{Definition}
\theoremstyle{definition}
\newtheorem{rem}[thm]{\textbf{Remark}}
\newtheorem*{rmk*}{Remark}
\newtheorem*{fact*}{Fact}

\theoremstyle{proof}

\newcommand{\vol}{\textrm{vol}}
\newcommand{\norm}[1]{\left\Vert#1\right\Vert}
\newcommand{\snorm}[1]{\Vert#1\Vert}
\newcommand{\abs}[1]{\left\vert#1\right\vert}
\newcommand{\set}[1]{\left\{#1\right\}}
\newcommand{\brac}[1]{\left(#1\right)}
\newcommand{\scalar}[1]{\left \langle #1 \right \rangle}

\newcommand{\Real}{\mathbb{R}}

\newcommand{\eps}{\varepsilon}

\newcommand{\N}{\mathbb{N}} 
\newcommand{\I}{\mathcal{I}}

\newcommand{\Llambda}{\lambda^{\text{Lip}}}
\def \F {\mathcal{F}}
\def \D {\mathcal{D}}
\newcommand{\CD}{\text{CD}}
\newcommand{\Ric}{\text{Ric}}
\newcommand{\lambdanu}{\nu}

\newlength{\defbaselineskip}
\newcommand{\setlinespacing}[1]           {\setlength{\baselineskip}{#1 \defbaselineskip}}

\numberwithin{equation}{section}

\begin{document}

\title{Spectral Estimates, Contractions and Hypercontractivity}
\author{Emanuel Milman\textsuperscript{1}}
\date{}

\footnotetext[1]{Department of Mathematics,
Technion - Israel Institute of Technology, Haifa 32000, Israel. Supported by ISF (grant no. 900/10), BSF (grant no. 2010288) and Marie-Curie Actions (grant no. PCIG10-GA-2011-304066).
The research leading to these results is part of a project that has received funding from the European Research Council (ERC) under the European Union's Horizon 2020 research and innovation programme (grant agreement No 637851).
Email: emilman@tx.technion.ac.il.}

\maketitle

\begin{abstract}
Sharp comparison theorems are derived for all eigenvalues of the weighted Laplacian, for various classes of weighted-manifolds (i.e. Riemannian manifolds endowed with a smooth positive density). Examples include Euclidean space endowed with strongly log-concave and log-convex densities, extensions to $p$-exponential measures, unit-balls of $\ell_p^n$, one-dimensional spaces and Riemannian submersions. Our main tool is a general Contraction Principle for ``eigenvalues" on arbitrary metric-measure spaces. Motivated by Caffarelli's Contraction Theorem, we put forth several conjectures pertaining to the existence of contractions from the canonical sphere (and Gaussian space) to weighted-manifolds of appropriate topological type having (generalized) Ricci curvature positively bounded below; these conjectures are consistent with all known isoperimetric, heat-kernel and Sobolev-type properties of such spaces, and would imply sharp conjectural spectral estimates. While we do not resolve these conjectures for the individual eigenvalues, we verify their Weyl asymptotic distribution in the compact and non-compact settings, obtain non-asymptotic estimates using the Cwikel--Lieb--Rozenblum inequality, and estimate the trace of the associated heat-kernel assuming that the associated heat semi-group is hypercontractive. As a side note, an interesting trichotomy for the heat-kernel is obtained. 
\end{abstract}

\section{Introduction}

A weighted-manifold is a triplet $(M^n,g,\mu)$, where $(M^n,g)$ is a complete smooth $n$-dimensional Riemannian manifold, endowed with a measure $\mu = \exp(-V(x)) d\vol_g(x)$ having smooth positive density with respect to the Riemannian volume measure $\vol_g$. The manifold can be compact or non-compact, but for simplicity we assume it is without boundary. In addition, there is no restriction on the total mass of the measure $\mu$. The associated weighted Laplacian $\Delta_{g,\mu}$ is defined as:
\[
\Delta_{g,\mu} f := \exp(V) \nabla_g \cdot (\exp(-V)  \nabla_g f) = \Delta_g f - \scalar{\nabla_g V,\nabla_g f} ,
\] 
so that the usual integration by parts formula is satisfied with respect to $\mu$:
\[
\forall f,h \in C_c^\infty(M) \;\;\; - \int_M (\Delta_{g,\mu} f)  h d\mu = \int_M \scalar{\nabla_g f , \nabla_g h} d\mu = - \int_M f (\Delta_{g,\mu} h) d\mu .
\]
Here $\nabla_g$ denotes the Levi-Civita connection, $\Delta_g$ denotes the usual Laplace-Beltrami operator, and we use $\scalar{\cdot,\cdot} = g$. 
One immediately sees that $-\Delta_{g,\mu}$ is a symmetric and positive semi-definite linear operator on $L^2(\mu)$ with dense domain $C_c^\infty(M)$, the space of compactly supported smooth functions on $M$. In fact, it is well-known  (e.g. \cite[Proposition 3.2.1]{BGL-Book}) 
that the completeness of $(M,g)$ ensures that $-\Delta_{g,\mu}$ is essentially self-adjoint on the latter domain, and so its graph-closure is its unique self-adjoint extension. We continue to denote the resulting positive semi-definite self-adjoint operator by $-\Delta_{g,\mu}$, with corresponding domain $Dom(\Delta_{g,\mu})$. 
By the spectral theory of self-adjoint operators (see Subsection \ref{subsec:manifold-contraction}), the spectrum $\sigma(-\Delta_{g,\mu})$ is a subset of $[0,\infty)$. When the spectrum is discrete (such as for compact manifolds), it is composed of isolated eigenvalues of finite multiplicity which increase to infinity; we denote these by $\lambda_k = \lambda_k(M^n,g,\mu)$, and arrange them in non-decreasing order (repeated by multiplicity) $0 \leq \lambda_1 \leq \lambda_2 \leq \ldots $.
In the discrete case, when $M$ is connected and $\mu$ has finite mass, we always have $0 = \lambda_1 < \lambda_2$.
For the standard definition of $\set{\lambda_k}$ when the spectrum is possibly non-discrete (as the first eigenvalues until the bottom of the essential spectrum), we refer to Subsection \ref{subsec:manifold-contraction}.
In this work, we would like to investigate the spectrum of various classes of weighted-manifolds. 

\begin{dfn*}
The weighted-manifold $(M^n,g,\mu)$ satisfies the Curvature-Dimension condition $\CD(\rho,N)$, $\rho \in \Real$ and $N=\infty$, if:
\begin{equation} \label{eq:Ric-Tensor}
\Ric_{g,\mu} := \Ric_g + \nabla^2_g V \geq \rho \; g 
\end{equation}
as symmetric 2-tensors on $M$. Here $\Ric_g$ denotes the Ricci curvature tensor, and $\Ric_{g,\mu}$ is called the generalized (infinite-dimensional) Ricci tensor. 
In this work, we restrict the Curvature-Dimension condition to connected manifolds. 
\end{dfn*}

The generalized Ricci tensor (\ref{eq:Ric-Tensor}) was introduced by Lichnerowicz \cite{Lichnerowicz1970GenRicciTensorCRAS,Lichnerowicz1970GenRicciTensor}, and extended to arbitrary generalized dimension $N \in (-\infty,\infty]$ by Bakry \cite{BakryStFlour} (cf. Lott \cite{LottRicciTensorProperties}). Note that in the constant density case $V \equiv c$, the generalized Ricci tensor boils down to the classical one. The Curvature-Dimension condition was introduced by Bakry and \'Emery in equivalent form in \cite{BakryEmery} (in the more abstract framework of diffusion generators). Its name stems from the fact that the generalized Ricci tensor incorporates information on curvature and dimension from both the geometry of $(M,g)$ and the measure $\mu$, and so $\rho$ may be thought of as a generalized-curvature lower bound, and $N$ as a generalized-dimension upper bound. 
With the exception of this Introduction, we will always assume that $N=\infty$ whenever referring to the Curvature-Dimension condition, so we omit the more general definition involving an arbitrary $N$ (cf. \cite{EMilmanSharpIsopInqsForCDD}). The $\CD(\rho,N)$ condition has been an object of extensive study over the last two decades (see e.g. also \cite{QianWeightedVolumeThms,LedouxLectureNotesOnDiffusion,CMSInventiones, CMSManifoldWithDensity, VonRenesseSturmRicciChar,BakryQianGenRicComparisonThms,WeiWylie-GenRicciTensor,MorganBook4Ed,EMilmanSharpIsopInqsForCDD,KolesnikovEMilman-Reilly, BGL-Book, KlartagLocalizationOnManifolds, EMilmanNegativeDimension} and the references therein), especially since Perelman's work on the Poincar\'e Conjecture \cite{PerelmanEntropyFormulaForRicciFlow}, and the extension of the Curvature-Dimension condition to the metric-measure space setting by Lott--Sturm--Villani \cite{SturmCD12,LottVillaniGeneralizedRicci}.

\subsection{Spectrum Comparison for Positively Curved Weighted-Manifolds} \label{subsec:intro-positive}

Let $\gamma^n_\rho$ denote the $n$-dimensional Gaussian probability measure with covariance $\frac{1}{\rho} Id$, namely $c^n_{\rho} \exp(-\rho \abs{x}^2/2 ) dx$, where $c_{\rho} > 0$ is a normalization constant. When $\rho = 1$, we simply write $\gamma^n$ for the standard $n$-dimensional Gaussian measure. 
It is well known \cite{Ledoux-Book,BGL-Book} that the one-dimensional Gaussian space $(\Real,\abs{\cdot},\gamma^1_\rho)$ serves as a model comparison space for numerous functional inequalities (such as isoperimetric \cite{BakryLedoux}, log-Sobolev \cite{BakryEmery} and spectral-gap $\lambda_2$ \cite{BGL-Book}), for the class of connected weighted-manifolds $(M^n,g,\mu)$ satisfying $\CD(\rho,\infty)$ with $\rho > 0$ (``positively curved weighted-manifolds"). The starting point of this work was to explore the possibility that these classical comparison properties also extend to all higher-order eigenvalues of $-\Delta_{g,\mu}$. Contrary to many functional inequalities, which remain invariant under tensorization, thus implying that the comparison space may be chosen to be one-dimensional,  the spectrum tensorization property naturally forces us to compare $(M^n,g,\mu)$ to the $n$-dimensional space $(\Real^n,\abs{\cdot},\gamma^n_\rho)$ having the same (topological) dimension. Note that positively curved weighted-manifolds always have discrete spectrum, since they necessarily satisfy a log-Sobolev inequality  by the Bakry--\'Emery criterion \cite{BakryEmery}, and the latter is known to imply (in our finite-dimensional setting!) discreteness of spectrum (see e.g. \cite{MazyaBook,WangSuperPoincareAndIsoperimetry, CianchiMazya-DiscretenessOfSpectrum}). In addition, if $(M^n,g,\mu)$ is positively curved then necessarily $\mu$ has finite total-mass \cite[Theorem 3.2.7]{BGL-Book}.

\begin{ques}[Spectral Comparison Question]  \label{ques:spectrum}
Given an $n$-dimensional connected weighted-manifold $(M^n,g,\mu)$ satisfying $\CD(\rho,\infty)$ with $\rho > 0$, 
does it hold that:
\[
\forall k \geq 1 \;\;\; \lambda_k(M^n,g,\mu) \geq \lambda_k(\Real^n,\abs{\cdot},\gamma^n_\rho) \;\;\; ?
\]
\end{ques}

At first, this question may seem extremely bold and at the same time classical and well-studied. As for the latter impression, we are not aware of any previous instances of Question \ref{ques:spectrum}.
The former impression perhaps stems from the extensive body of work in trying to just provide sharp lower and upper bounds on the first eigenvalue gap $\lambda_2 - \lambda_1$ under various conditions (e.g. \cite{KLS,SpanishBook,AshbaughBenguria-EigenvaluesSurvey,BenguriaLindeLoewe-EigenvaluesSurvey,AndrewsClutterbuck-FundamentalGapForSchrodinger}), or various other conjectured lower bounds on the entire spectrum, such as Polya's conjecture (see e.g. \cite{ChavelEigenvalues,LiYau-PolyaAndCLR, Laptev-LiebThirringSurvey}). 

\medskip

Unfortunately, one cannot expect to have a positive answer to the above question in general, at least not for the first eigenvalues. The easiest counterexample is given by the canonical $n$-sphere, rescaled to have Ricci curvature equal to $1$ (times the metric), so that it satisfies $\CD(1,\infty)$; its $(n+2)$-th eigenvalue (given by a linear function on the sphere's canonical embedding in $\Real^{n+1}$) is equal to $\frac{n}{n-1}$, whereas the corresponding eigenvalue for the $n$-dimensional Gaussian space is already equal to $2$ (see Subsection \ref{subsec:sphere} for more details). 

\medskip
Nevertheless, we can show:

\begin{thm}[Spectral Comparison for Positively Curved $(\Real^n,\abs{\cdot},\mu)$] \label{thm:spectrum} 
Question \ref{ques:spectrum} has a positive answer for any Euclidean space $(\Real^n,\abs{\cdot},\mu)$  satisfying $\CD(\rho,\infty)$ with $\rho > 0$. 
\end{thm}

In view of the above counterexample and theorem, and for reasons which will become more apparent later on, it is plausible that some topological restrictions must be enforced to obtain a positive answer to Question \ref{ques:spectrum}. The simplest one is to assume that $(M^n,g)$ is diffeomorphic to Euclidean space. We tentatively formulate this as:
\begin{conj1}[Spectral Comparison Conjecture for Positively Curved $(\Real^n,g,\mu)$] Question \ref{ques:spectrum} has a positive answer for any $(\Real^n,g,\mu)$ satisfying $\CD(\rho,\infty)$ with $\rho > 0$. 
\end{conj1}

See Section \ref{sec:conclude} for a more refined version of this tentative conjecture. Clearly, the case $n=1$ boils down to Theorem \ref{thm:spectrum}, so the conjecture pertains to the range $n \geq 2$. We take this opportunity to also mention the work of Ledoux \cite{Ledoux-GammaAlgebra} (cf. Bakry and Bentaleb \cite{BakryBentaleb}), who showed how information on higher order iterated \emph{carr\'e-du-champ} operators (so called $\Gamma_k$ operators) may be used to obtain higher-order eigenvalue estimates for the generator $-\Delta_{g,\mu}$; however, here we only assume the $\CD(\rho,\infty)$ condition, which amounts to information on $\Gamma_2$ only (see \cite{BGL-Book} for more on $\Gamma$-Calculus).

\subsection{Spectrum Comparison for Additional Spaces}

Our method of proof of Theorem \ref{thm:spectrum}, described in the next subsection, is very general, and in particular also yields the following additional results:

\begin{thm} \label{thm:Kolesnikov}
Let $(\Real^n,\abs{\cdot},\mu)$ denote a Euclidean weighted-manifold where $\mu = \exp(-V(x)) dx$ is a probability measure satisfying $\nabla^2 V \leq \rho Id$. Then:
\[
\forall k \geq 1 \;\;\; \lambda_k(\Real^n,\abs{\cdot},\mu) \leq \lambda_k(\Real^n,\abs{\cdot},\gamma^n_\rho) . 
\]
\end{thm}

\begin{thm} \label{thm:KM}
Assume $p \in [1,2]$, and set $\nu^{n}_p := c_p^n \exp(-\sum_{i=1}^n \frac{1}{p} \abs{x_i}^p ) dx$ (with $c_p > 0$ chosen so that $\nu^n_p$ is a probability measure) . Let $\mu = \exp(-U) \nu^n_p$ denote a second probability measure on $\Real^n$, and assume that $U : \Real^n \rightarrow \Real$ is convex and unconditional, meaning that $U(\pm x_1, \ldots, \pm x_n) = U(x)$. Then:
\[
\forall k \geq 1 \;\;\; \lambda_k(\Real^n,\abs{\cdot},\mu) \geq \lambda_k(\Real^n,\abs{\cdot},\nu^n_p) . 
\]
\end{thm}

\begin{thm} \label{thm:Latala} 
Assume $p \in [2,\infty]$, and let $\tilde{B}_p^n$ denote the unit ball of $\ell_p^n$, rescaled to have volume 1; the uniform measure on $\tilde{B}_p^n$ is denoted by $\lambdanu_{\tilde{B}_p^n}$. Then:
\[
\forall k \geq 1 \;\;\; \lambda_k(\Real^n,\abs{\cdot},\lambdanu_{\tilde{B}_p^n}) \geq \frac{1}{392} \lambda_k(\Real^n,\abs{\cdot},\gamma^{n}) . 
\]
Here $\lambda_k(\Real^n,\abs{\cdot},\lambdanu_{\tilde{B}_p^n})$ denote the eigenvalues of $-\Delta$ on $\tilde{B}_p^n$ with vanishing Neumann boundary conditions. 
\end{thm}

\begin{thm} \label{thm:Bobkov}
Given a weighted-manifold $(\Real,\abs{\cdot},\mu)$ with $\mu$ a probability measure, denote its density by $f_\mu$,  by $F_\mu(x) = \mu((-\infty,x])$ its cumulative distribution function, and set $\I^\flat_\mu := f_\mu \circ F_\mu^{-1} : [0,1] \rightarrow \Real_+$ (its ``one-sided flat isoperimetric profile"). Let $\mu_1,\mu_2$ denote two such measures. Then:
\[
\I^\flat_{\mu_2} \geq \frac{1}{L} \I^\flat_{\mu_1}   \text{ on $[0,1]$} \;\;\; \Rightarrow \;\;\; \forall k \geq 1 \;\;\; \lambda_k(\Real,\abs{\cdot},\mu_2) \geq \frac{1}{L^2}  \lambda_k(\Real,\abs{\cdot},\mu_1)  .
\]
\end{thm}

Recall that a map $T : (M^{n_1}_1,g_1) \rightarrow (M^{n_2}_2,g_2)$ is called a Riemannian submersion if $T$ is smooth, surjective (so that $n_1 \geq n_2$), and at every point $x \in M_1$, the differential $d_x T$ is of maximal rank $n_2$ and an isometry on the orthogonal complement to its kernel.

\begin{thm} \label{thm:submersion}
Let $(M_i,g_i,\mu_i)$ denote two weighted-manifolds, and let $T : (M_1,g_1) \rightarrow (M_2,g_2)$ denote a Riemannian submersion pushing forward $\mu_1$ onto $\mu_2$ up to a finite constant. Then:
\[
\forall k \geq 1 \;\;\; \lambda_k(M_2,g_2,\mu_2) \geq \lambda_k(M_1,g_1,\mu_1) . 
\]
In particular, this holds for $\mu_i = \vol_{g_i}$, the corresponding Riemannian volume measures, if both manifolds are connected and the submersion's fibers are minimal and compact. \\
In particular, this holds for any finite-sheeted Riemannian covering map $T$ between two connected manifolds. 
\end{thm}

\noindent
We refer to Subsection \ref{subsec:submersions} for more background on Riemannian submersions and a slightly more general result. The ``in particular" parts of Theorem \ref{thm:submersion} are certainly not new (at least when the manifolds are compact, see e.g. \cite[Section 3]{Bordoni-HeatEstimatesForSubmersions}).

\subsection{Contracting and Lipschitz Maps}

Let $T : (\Omega_1,d_1,\mu_1) \rightarrow (\Omega_2,d_2,\mu_2)$ denote a Borel map between two metric-measure spaces. The map $T$ is said to push-forward the probability measure $\mu_1$ onto $\mu_2$, denoted $T_*(\mu_1) = \mu_2$, if $\mu_2 = \mu_1 \circ T^{-1}$. To treat the case when $\mu_i$ may have different or infinite total mass, we will say that $T$ pushes forward $\mu_1$ onto $\mu_2$ up to a finite constant, if there exists $c \in (0,\infty)$ so that $T$ pushes forward $\mu_1$ onto $c \mu_2$. The map $T$ is called $L$-Lipschitz ($L > 0$) if:
\[
d_2(T(x),T(y)) \leq L \; d_1(x,y) \;\;\; \forall x,y \in \Omega_1 .
\]
The map $T$ is called a contraction if it is Lipschitz with constant $L=1$. 

\medskip

All of our spectrum comparison theorems are consequences of the following:
\begin{thm}[Contraction Principle] \label{thm:intro-Lipschitz}
Let $T : (M_1,g_1,\mu_1) \rightarrow (M_2,g_2,\mu_2)$ denote an $L$-Lipschitz map between two (complete) weighted-manifolds pushing-forward $\mu_1$ onto $\mu_2$ up to a finite constant. Then:
\[
\lambda_k(M_2,g_2,\mu_2) \geq \frac{1}{L^2} \lambda_k(M_1,g_1,\mu_1) \;\;\; \forall k \geq 1. 
\]
In particular, if $\Delta_{g_1,\mu_1}$ has discrete spectrum, then so does $\Delta_{g_2,\mu_2}$.
\end{thm}

In fact, an analogous result holds for compact weighted-manifolds with either Dirichlet or Neumann boundary conditions, see Subsection \ref{subsec:boundary}. 
We note that even in the classical non-weighted setting, the contraction principle is easily seen to be completely false if we omit the assumption that $T$ pushes forward the first volume measure onto the second (up to a finite constant);
moreover, in that case, even if $T$ is known to be bi-Lipschitz, the resulting spectrum comparison would depend exponentially on the underlying dimension $n$, which is often useless for applications.

\smallskip
While the derivation of Theorem \ref{thm:intro-Lipschitz} is straightforward, we have not encountered an application of contracting maps for spectrum comparison elsewhere. Using a standard argument on the density of locally Lipschitz functions in $W^{1,2}(M,g,\mu)$, the statement of Theorem \ref{thm:intro-Lipschitz} is seen in Section \ref{sec:contractions} to be a particular case of an analogous statement on ``metric eigenvalue" comparison, which in fact holds true in a general metric-measure space setting - this is the content of Proposition \ref{prop:metric-contract}. Working with locally Lipschitz functions $f$ on $(M_2,g_2)$ is particularly convenient, as it decouples the discussion regarding domains of (essential) self-adjointness of $\Delta_{g,\mu}$, from the metric argument underlying the proof of the Contraction Principle:
\[
\frac{\int \abs{\nabla (f \circ T)}^2 d\mu_1}{\int \abs{f\circ T}^2 d\mu_1} \leq L^2 
\frac{\int \abs{\nabla f}^2 d\mu_2}{\int \abs{f}^2 d\mu_2} \; ;
\]
here $\abs{\nabla h}$ denotes the Riemannian length of the gradient of $h$ on the corresponding manifold (and the local Lipschitz constant of $h$ in a general metric setting). 
The proof of Theorem \ref{thm:intro-Lipschitz} (and of Proposition \ref{prop:metric-contract}) then follows by a careful application of the min-max principle. Two slightly delicate points here are that we do not assume injectivity of $T$ (which is useful for some of the applications above), and in the general metric setting, that we do not assume $\mu_i$ have full supports, and so the min-max principle should be carefully checked.  

\smallskip
Finally, we mention that the statement of Theorem \ref{thm:intro-Lipschitz} may be equivalently rewritten on $Dom(\Delta_{g_2,\mu_2}) \cap T_*(Dom(\Delta_{g_1,\mu_1}))$ as:
\begin{equation} \label{eq:op-comparison}
-\Delta_{g_2,\mu_2} \geq \frac{1}{L^2} (T^*)^t \circ (-\Delta_{g_1,\mu_1}) \circ T^* , 
\end{equation}
where $T_*$ and $T^*$ denote the push-forward and pull-back maps between $L^2(M_1,\mu_1)$ and $L^2(M_2,\mu_2)$ induced by $T$. To better appreciate the above stated comparison, the reader may wish to try and explicitly write out and compare the differential operators appearing in (\ref{eq:op-comparison}) using the change-of-variables formula relating $\mu_1$, $\mu_2$ and $\text{det}(dT)$ - this is quite a tedious task, which does not give any insight towards why (\ref{eq:op-comparison}) holds true.

\medskip

A few words are in order regarding previous approaches towards spectrum comparison between differential operators on Riemannian manifolds (and more generally, linear operators on Hilbert spaces). The closest general argument we have found in the literature is the so-called Kato's inequality and its generalizations (see \cite{HSU-DominationOfSemigroupsViaKato, Simon-KatoInequality, Besson-KatoTypeInqForTotallyGeodesicSubmersions, 
Besson-BeurlingDenyPrinciple, Berard-LecturesOnSpectralGeometry, Bordoni-SpectralSurvey,Bordoni-SpectralEstimatesViaKato,Bordoni-HeatEstimatesForSubmersions} and the references therein), which under certain conditions permit comparing the trace of the associated heat semi-groups, heat-kernels, and even the heat semi-group and resolvent operators themselves in the sense of domination of positivity preserving operators. However, these results typically do not involve the individual eigenvalues (cf. \cite[III.6]{Besson-BeurlingDenyPrinciple}), and in the few cases that do, the conclusion is in the opposite direction to the one appearing in this work (in an attempt to obtain spectral lower bounds on the source manifold by mapping it onto a simpler one). 
We also mention two additional classical methods of obtaining estimates on the growth and number of negative eigenvalues of a Schr\"{o}dinger operator -- the Lieb--Thirring and Cwikel--Lieb--Rozenblum inequalities \cite{LiebThirring,Cwikel-CLR,Lieb-CLR-Announced,Rozenblum-CLR} (see also \cite{Laptev-LiebThirringSurvey} and the references therein), the latter of which we will 
in fact employ in this work as well (see Subsection \ref{subsec:CLR}). 

\medskip

Back to the Contraction Principle. 
A celebrated contraction property was discovered by L. Caffarelli in \cite{CaffarelliContraction}:
\begin{thm*}[Caffarelli's Contraction Theorem]
Let $(\Real^n,\abs{\cdot},\mu)$ satisfy $\CD(\rho,\infty)$ with $\rho > 0$. Then there exists a map $T : (\Real^n,\abs{\cdot}) \rightarrow (\Real^n,\abs{\cdot})$ pushing forward $\gamma^n_\rho$ onto $\mu$ up to a finite constant which contracts Euclidean distance.
\end{thm*}
Together with the Contraction Principle, this immediately yields Theorem \ref{thm:spectrum}. 
Caffarelli proved the above result for the Brenier Optimal-Transport map $T$ \cite{VillaniTopicsInOptimalTransport}, which uniquely (up to a null-set deformation) minimizes the $L^2$-averaged transport distance $\int \abs{T(x) - x}^2 d\gamma^n_\rho(x)$ among all maps  pushing forward $\gamma^n_\rho$ onto $\mu$. Subsequently in \cite{KimEMilmanGeneralizedCaffarelli}, Young-Heon Kim and the author gave an alternative proof and extended Caffarelli's theorem using a (seemingly) different map $T$ involving a naturally associated heat-flow, which together with the Contraction Principle immediately yields Theorem \ref{thm:KM}. Similarly, the existence of contracting and Lipschitz maps due to Kolesnikov \cite{KolesnikovContractionSurvey}, Lata{\l}a--Wojtaszczyk \cite{LatalaJacobInfConvolution} and Bobkov--Houdr\'e \cite{BobkovHoudre} yield Theorems \ref{thm:Kolesnikov}, \ref{thm:Latala} and \ref{thm:Bobkov}, respectively; details are provided in Section \ref{sec:known-contractions}.

\medskip

Contracting, and more generally, Lipschitz maps between metric-measure spaces, constitute a very powerful tool for transferring isoperimetric, functional and concentration information from $(\Omega_1,d_1,\mu_1)$ to $(\Omega_2,d_2,\mu_2)$. However, for these traditional applications, there are numerous other tools available, such as $\Gamma_2$-Calculus, other parabolic and elliptic $L^2$-methods, Optimal-Transport, Localization, etc.. (see e.g. \cite{BGL-Book,KolesnikovEMilman-Reilly,KlartagLocalizationOnManifolds}). As shown in this work, contracting maps also yield sharp comparison estimates for the entire spectrum, going well beyond the capability of the above mentioned alternative methods - we believe this to be a noteworthy (albeit simple) observation. 

Motivated by Caffarelli's Contraction Theorem on one hand, and the well-known comparison results between weighted-manifolds satisfying $\CD(\rho,\infty)$  and the ($1$ or equivalently $n$-dimensional) Gaussian measure $\gamma_\rho$ ($\rho > 0$) on the other, we tentatively put forth the following conjecture, which by the Contraction Principle, would imply Conjecture \conjone:

\begin{conj2}[Contraction Conjecture for Positively Curved $(\Real^n,g,\mu)$] For any $(\Real^n,g,\mu)$ satisfying $\CD(\rho,\infty)$ with $\rho > 0$, there exists a map:
\[
T : (\Real^n,\abs{\cdot},\gamma^n_\rho) \rightarrow (\Real^n, g , \mu) ,
\]
pushing forward $\gamma^n_\rho$ onto $\mu$ up to a finite constant and contracting the corresponding metrics. 
\end{conj2}

See Section \ref{sec:conclude} for a more refined version of this tentative conjecture. 
Conjecture \conjtwo is consistent with the Bakry-Ledoux isoperimetric comparison theorem \cite{BakryLedoux} and the Bakry-\'Emery log-Sobolev inequality \cite{BakryEmery} for $\CD(\rho,\infty)$ weighted-manifolds. Note that we have restricted the above conjecture to manifolds diffeomorphic to $\Real^n$, as the counterexample of the canonical sphere from Subsection \ref{subsec:intro-positive} shows that one cannot hope for such a map unto a general weighted-manifold $(M^n,g,\mu)$ satisfying $\CD(\rho,\infty)$. Moreover, there are topological obstructions to the existence of such a map between $(\Real^n,\abs{\cdot})$ and $(M^n,g)$, at least if we assume in addition that $T$ is one-to-one from the source onto the target manifold: indeed, Brouwer's Invariance of Domain theorem \cite{Munkres-TopologyBook2ndEd} asserts that an injective, surjective and continuous map between two topological manifolds is in fact open, and hence the two manifolds must be homeomorphic. 

For a further discussion and refinement of Conjectures \conjone and \conjtwo, we refer to Section \ref{sec:conclude}.

\subsection{Extensions to Positively Curved Constant-Density Manifolds}

It is of course very natural to attempt extending the previous conjectures to the class of weighted-manifolds satisfying $\CD(\rho,N)$ for $\rho > 0$ and finite generalized dimension $N \in [n,\infty)$. Contrary to the situation with the usual functional inequalities (isoperimetric, Sobolev, spectral-gap, cf. \cite{BakryQianSharpSpectralGapEstimatesUnderCDD,EMilmanSharpIsopInqsForCDD}), it is not so clear  what would be the right (topologically $n$-dimensional) model space for comparing the entire spectrum. However, when $N=n$, which corresponds to the classical case of a complete connected Riemannian manifold, endowed with its canonical Riemannian volume measure $\vol_g$ and having Ricci curvature bounded below by $\rho > 0$ (times the metric), the natural model space is simply the canonical $n$-sphere $S^n$ with its metric $g^\rho_{can}$ rescaled to have $\Ric_{g^\rho_{can}} = \rho \; g^\rho_{can}$. For similar topological reasons as in the previous subsection (see also the ensuing discussion), we restrict to the case when $(M^n,g)$ is diffeomorphic to a sphere. 

\setcounter{conj}{2}
\begin{conj}[Spectral Comparison Conjecture For Positively Curved $(S^n,g,\vol_g)$] \label{conj:3}
For any $(S^n,g,\vol_g)$ satisfying $\Ric_{g} \geq \rho g$ with $\rho > 0$, we have:
\[
\forall k \geq 1 \;\;\; \lambda_k(S^n,g,\vol_g) \geq \lambda_k(S^n,g^\rho_{can},\vol_{g^\rho_{can}}) . 
\] 
\end{conj}

\noindent
Conjecture \ref{conj:3} is consistent with:
\begin{itemize}
\item The Lichnerowicz spectral-gap estimate $\lambda_2(M^n,g,\vol_g) \geq \lambda_2(S^n,g^\rho_{can},\vol_{g^\rho_{can}}) = \rho \frac{n}{n-1}$ \cite{LichnerowiczBook}.
\item The B\'erard--Gallot estimate on the trace of the heat-kernel \cite{BerardGallot-HeatEquationComparison}:
\begin{equation} \label{eq:BG-trace}
\forall t > 0 \;\;\; \sum_{k=1}^\infty \exp(-t \lambda_k(S^n,g,\vol_g)) \leq \sum_{k=1}^\infty \exp(-t \lambda_k(S^n,g^\rho_{can},\vol_{g^\rho_{can}}) ) .
\end{equation}
\item It is immediate to show that it is compatible with Weyl's asymptotic law -- see Subsection \ref{subsec:Weyl}.
\item We can actually show that it holds true up to a dimension \emph{independent} multiplicative constant for $k \geq 4^n$ -- see Subsection \ref{subsec:CLR}.
\end{itemize}

Conjecture \ref{conj:3} would follow immediately from the Contraction Principle and the following previously unpublished conjecture of ours \cite{EMilman-GIF-2009}:

\begin{conj}[Contraction Conjecture for Positively Curved $(S^n,g,\vol_g)$] \label{conj:4}
For any $(S^n,g,\vol_g)$ satisfying $\Ric_{g} \geq \rho g$ with $\rho > 0$, there exists a map:
\[
T : (S^n,g^\rho_{can},\vol_{g^\rho_{can}}) \rightarrow (S^n,g,\vol_g) ,
\]
pushing forward $\vol_{g^\rho_{can}}$ onto $\vol_g$ up to a finite constant and contracting the corresponding metrics. 
\end{conj}

Note that a connected complete Riemannian manifold $(M,g)$ with $\Ric_g \geq \rho g$, $\rho > 0 $, is necessarily compact and has finite volume.
The reader should note the apparent analogy between the latter conjecture and Caffarelli's Contraction Theorem, in view of  the definition of the generalized Ricci tensor (\ref{eq:Ric-Tensor}). Conjecture \ref{conj:4} is consistent with:
\begin{itemize}
\item The Bonnet--Meyers bound on the diameter of such manifolds \cite{GHLBookEdition3}:
\[
diam(M) \leq diam(S^n,g^\rho_{can}) = \frac{\pi}{\sqrt{\rho / (n-1)}} .
\]
\item The Bishop--Gromov volume estimate \cite{GHLBookEdition3}:
\[
\frac{\vol_g(B(x_0,r))}{\vol_g(M^n)} \geq \frac{\vol_{g^\rho_{can}}(B(x_1,r))}{\vol_{g^\rho_{can}}(S^n)} \;\; \forall x_0 \in M^n, x_1 \in S^n \;\; \forall r > 0 ,
\]
which in particular implies (letting $r\rightarrow 0$) $\vol_g(M^n) \leq \vol_{g^\rho_{can}}(S^n)$. 
\item
The Bakry--\'Emery log-Sobolev estimate $L(M^n,g,\vol_g) \geq L(S^n,g^\rho_{can},\vol_{g^\rho_{can}}) =\rho \frac{n}{n-1}$ \cite{BakryEmery} (see Section \ref{sec:semigroup}).
\item
The sharp Sobolev inequality for $\CD(\rho,n)$ spaces \cite[Theorem 6.8.3]{BGL-Book}.
 \item
The Gromov--L\'evy isoperimetric inequality \cite{GromovGeneralizationOfLevy}.
\item
Conjecture \ref{conj:3} on the full spectrum, including all of the known consequences mentioned after its formulation above. \end{itemize}
A positive answer to Conjecture \ref{conj:4} would thus yield a single reason to all of these classical facts (albeit only for manifolds which are diffeomorphic to a sphere). It would be very interesting to adapt and extend the Optimal-Transport or Heat-Flow approaches of Caffarelli \cite{CaffarelliContraction} and Kim and the author \cite{KimEMilmanGeneralizedCaffarelli} from the scalar setting (involving densities) to the above $2$-tensorial setting (involving metrics) - cf. \cite{EMilman-GIF-2009}.

As before, we have restricted Conjecture \ref{conj:4} to $M=S^n$ due to potential topological obstructions. Indeed, a map $T : (S^n,g^\rho_{can},\vol_{g^\rho_{can}}) \rightarrow (M^n,g,\vol_g) $ as in Conjecture \ref{conj:4} must be surjective, since $T(S^n) \subset M^n$ is compact as a continuous image of a compact set, while its open complement satisfies $\vol_g(M^n \setminus T(S^n)) = 0$, and hence must be empty. Consequently, if we assume in addition that $T$ is injective, Brouwer's Invariance of Domain theorem would imply as before that $T$ is open, and hence $M$ must be homeomorphic to $S^n$. 

For simplicity, we have chosen not to explicitly formulate the most general possible conjectures in the above spirit. Let us only remark that if we do not insist on finding a \emph{topologically $n$-dimensional} model source space which conjecturally contracts onto $\CD(\rho,N)$ $n$-dimensional weighted-manifolds ($\rho > 0$), thereby giving up on obtaining asymptotically sharp eigenvalue estimates (per Weyl's law) and on injectivity of the contracting map, then a reasonable choice for such a model source space, 
at least when $N > n$ is an integer, is the rescaled canonical $N$-sphere; this would still be consistent with all known generalizations of the above properties (see \cite{BGL-Book, SturmCD12, EMilmanSharpIsopInqsForCDD} and the references therein), and contrary to the counterexample of Subsection \ref{subsec:intro-positive}, is easily verified for $(M^n,g) = (S^n,g_{can}^\rho)$. 
It is also possible to consider adding the case $\rho = 0$ to the above setting (under suitable modifications, replacing $S^n$ with $\Real^n$), but we do not have a clear sense of how reasonable this might be.

\subsection{Comparison on Average}

While we were not able to resolve Conjectures \conjone nor \ref{conj:3}, we would still like to mention some tools for controlling the eigenvalues \emph{in some averaged sense}. In Subsection \ref{subsec:Weyl}, we recall Weyl's asymptotic law for the distribution of eigenvalues in the compact case, and develop its analog in the weighted non-compact setting. However, we would like to obtain some concrete non-asymptotic estimates as well. 

In Subsection \ref{subsec:CLR}, we show that Conjecture \ref{conj:3} is satisfied up to a dimension independent multiplicative constant for exponentially large (in the dimension) eigenvalues, by making use of the classical Cwikel--Lieb--Rozenblum inequality together with the sharp Sobolev inequality on $\CD(\rho,n)$ weighted manifolds. We did not manage to verify a similar conclusion for Conjecture \conjone, perhaps because the $\CD(\rho,\infty)$ condition does not directly feel the dimension $n$. We therefore proceed to obtain some average estimates for the eigenvalues. 

When $\mu$ is a probability measure, a very natural function encapsulating the growth of the eigenvalues is given by the trace of the heat semi-group $P_t = \exp(t \Delta_{g,\mu})$:
\[
Z(t) = Z_{(M^n,g,\mu)}(t) := \text{tr}(P_t) = \int p_t(x,x) d\mu(x) = \sum_{k=1}^\infty \exp(- t \lambda_k) ,
\]
where $p_t(x,y) = p_t(M^n,g,\mu)(x,y)$ denotes the heat-kernel (with respect to $\mu$). 
It is an interesting question to establish conditions on $(M^n,g,\mu)$ which ensure that $P_t$ is trace-class, i.e. that $Z(t) < \infty$ for $t > t_0 \geq 0$. In particular, upper bounds on $Z(t)$ yield lower bounds on the individual eigenvalues by the trivial estimate:
\begin{equation} \label{eq:Z-to-lambda}
\forall k \geq 1 \;\;\; k \exp(-t \lambda_k) \leq Z(t) . 
\end{equation}
However, it may very well happen that the spectrum is discrete (equivalently, that $\lambda_k$ increase to infinity), and yet $Z(t) = \infty$ for all $t > 0$. 
Note that $Z(t)$ will inevitably \emph{depend on the dimension} $n$, e.g. because of Weyl's law or because of the spectrum's tensorization property - see Section \ref{sec:eigenvalues} for concrete examples such as for the $n$-dimensional Gaussian space or sphere. This is in contrast to more traditional objects of study on weighted-manifolds (such as the spectral-gap or log-Sobolev constant), which are invariant under tensorization, and thus often dimension-independent. 

\medskip

In connection to the discussion regarding previously known estimates on the spectrum, we mention the following result of B\'erard and Gallot \cite{BerardGallot-HeatEquationComparison,Berard-LecturesOnSpectralGeometry} (see also Besson \cite[Appendix]{Besson-BeurlingDenyPrinciple}). By employing the Gromov--L\'evy isoperimetric inequality \cite{GromovGeneralizationOfLevy}, these authors showed that for any connected $(M^n,g,\tilde{\vol_g})$ with $\Ric_{g} \geq \rho g$, $\rho >0$, one has:
\[
\sup_{x \in M} p_t(M^n,g,\tilde{\vol}_g)(x,x) \leq  p_t(S^n,g^\rho_{can},\tilde{\vol}_{g^\rho_{can}})(x_0,x_0) \;\;\; \forall t>0 \;,\; x_0 \in S^n \; ,
\]
where $\tilde{\mu} = \mu / \norm{\mu}$ denotes $\mu$ renormalized to be a probability measure. In particular, this yields (\ref{eq:BG-trace}), 
confirming Conjecture \ref{conj:3} in an averaged (yet strictly weaker) sense. By employing the Bakry--Ledoux isoperimetric inequality \cite{BakryLedoux}, it may also be possible to obtain a somewhat similar \emph{on-average} confirmation of Conjecture \conjone; this is not immediate and will be explored elsewhere. Here, we are more interested in another direction. 

\medskip

Upper bounds on $Z(t)$ and moreover lower bounds on $\lambda_k$ under various assumptions on  $(M^n,g,\mu)$ were obtain by F.-Y. Wang in \cite{Wang-EigenvalueEstimates,Wang-SurveyOnSemiGroupAndSpectrum} (we refer to the excellent book \cite{BGL-Book} and to Section \ref{sec:semigroup} for subsequent missing references and terminology, which we only mention here in passing). Whenever the space satisfies a Sobolev inequality (or equivalently a Nash inequality, or finite-dimensional log-Sobolev inequality \cite[Chapter 6]{BGL-Book}), and in particular under a $\CD(\rho,N)$ condition for $\rho > 0$ and finite $N \in [n,\infty)$, it is well-known that $P_t$ is ultracontractive \cite[Corollary 6.3.3.]{BGL-Book},  i.e. that the heat-kernel $p_t$ is bounded, yielding a trivial upper-bound on $Z(t)$. This ultracontractive case has been extensively studied in the literature, see e.g. \cite{DaviesSemiGroupBook,BGL-Book,Wang-EigenvalueEstimates}. The borderline case when some \emph{additional} information is needed is precisely when $P_t$ is only hypercontractive, i.e. when  $(M^n,g,\mu)$ satisfies a log-Sobolev inequality. In that case, and even under a weaker super-Poincar\'e (or $F$-Sobolev) inequality, assuming in addition that the space satisfies $\CD(\rho,\infty)$ for some $\rho \in \Real$, Wang obtained very general lower bounds on $\lambda_k$ depending on concentration properties of the distance function $d(x,x_0)$ to a given point $x_0 \in M$.

In Section \ref{sec:semigroup}, we expand on the quantitative relation between hypercontractivity of the heat semi-group, the property of being trace-class (i.e. upper estimates on $Z(t)$), and higher-order integrability properties of the associated heat-kernel, both in general and under a $\CD(\rho,\infty)$ condition, $\rho \in \Real$. Our approach closely follows Wang's method, based on his dimension-free Harnack inequality. Our results are weaker and less general than Wang's, but the proofs are a bit simplified, yielding estimates with concrete dimension-dependence. Finally, a general interesting trichotomy for the heat-kernel is deduced. In Section \ref{sec:conclude}, we provide some concluding remarks.

\medskip \noindent \textbf{Acknowledgements.} I thank Franck Barthe, Mikhail Gromov, Tobias Hartnick, Michel Ledoux and Yehuda Pinchover for their comments. I also thank the anonymous referee for helpful comments which have improved the exposition of the paper.

\section{Eigenvalue Calculation and Asymptotics} \label{sec:eigenvalues}

We begin with calculating the eigenvalues or their asymptotic distribution for several notable weighted-manifolds. 

\subsection{Gaussian Space} \label{subsec:Gaussian}

It is well known that the one-dimensional Gaussian Space $(\Real,\abs{\cdot},\gamma^1)$ has simple spectrum at $\N_0 := \set{0,1,2,3,\ldots}$ (so that each of the eigenvalues has multiplicity one), with the eigenfunctions of $-\Delta_{\gamma^1}$ being precisely the Hermite polynomials.  
By the tensorization property of the spectrum, it follows that the product space $(\Real^n,\abs{\cdot},\gamma^n)$ has spectrum $\N_0 + \ldots + \N_0$, where the sum is repeated $n$ times and is counted with multiplicity. In other words, the spectrum consists of $\N_0$ and the multiplicity of the eigenvalue $l \in \N_0$ is given by $n-1 + l \choose l$. It follows that the eigenvalue counting function satisfies for $\lambda \in \N_0$:
\[
\# \set{ \lambda_k \leq \lambda} = \sum_{l=0}^\lambda {n-1 + l \choose l} = {n+\lambda \choose \lambda} \leq \min \brac{ \brac{\frac{(n+\lambda) e}{n}}^n ,  \brac{\frac{(n+\lambda) e}{\lambda}}^\lambda}  ,
\]
 and consequently:
\begin{equation} \label{eq:Gaussian-eigenvalues}
\lambda_k \geq \max \brac{ \frac{n}{e} k^{1/n} -  n , \frac{\log k}{\log ((n+1) e)} } .
\end{equation}
Furthermore, we record that as $\lambda \rightarrow \infty$ we have:
\begin{equation} \label{eq:Gaussian-Theta}
\# \set{ \lambda_k \leq \lambda} =  \frac{\lambda^n}{\Gamma(n+1)} (1 + o(1)). 
\end{equation}

\subsection{Canonical Sphere} \label{subsec:sphere}

Let $(S^n,g_{can},\vol_{g_{can}})$ denote the $n$-sphere with its canonical metric and volume measure, embedded as the unit-sphere in Euclidean space $(\Real^{n+1},\abs{\cdot})$; its Ricci tensor is equal to $(n-1) g$. It is well known that the eigenfunctions of the associated Laplacian $-\Delta_{S^n}$ are given by spherical-harmonics, i.e. the restriction of harmonic homogeneous polynomials in $\Real^{n+1}$ onto $S^n$. The eigenvalue associated to harmonic polynomials $q_l(x)$ of degree $l \in \N_0$ is $l (l + n-1)$ \cite{VilenkinClassicBook}. Since it is well known that any homogeneous polynomial $p(x)$ of degree $m \in \N_0$ can be uniquely decomposed into its harmonic components as follows:
\[
p(x) = q_{m}(x) + \abs{x}^2 q_{m-2}(x) + \abs{x}^4 q_{m-4}(x) + \ldots ,
\]
we see that the subspace spanned by spherical harmonics of even degree at most $m \in 2 \N_0$ or of odd degree at most $m \in 2 \N_0 -1$ is of dimension $n + m \choose m$. Consequently, the subspace of all spherical harmonics of degree at most $m \in \N_0$ is of dimension ${n + m \choose m} + {n + m-1 \choose m-1}$, with corresponding eigenvalues being at most $m (m + n-1)$. 

Now let us rescale the canonical sphere to have radius $\sqrt{n-1}$, so that its Ricci tensor coincides with the metric and therefore satisfies $\CD(1,\infty)$. Since the eigenvalues scale quadratically in the metric, the eigenvalue counting function consequently satisfies for all $m \in \N_0$:
\[
\# \set{ \lambda_k \leq m \frac{m+n-1}{n-1} } = {n + m \choose m} + {n + m-1 \choose m-1} . 
\]
One could hope that the counting function of the rescaled $n$-sphere is always dominated by that of the $n$-dimensional Gaussian:
\[
 {n + m \choose m} + {n + m-1 \choose m-1} \leq_{?} {n + v  \choose v}  ~,~ v :=  \left \lfloor   m \frac{m+n-1}{n-1}  \right \rfloor = m +  \left \lfloor \frac{m^2}{n-1} \right \rfloor .
\]
However, this is \textbf{not the case} for the first eigenvalues, and is most apparent for $m=1$, i.e. linear functions on the sphere. Indeed, for all $n \geq 3$, $\lambda_{n+2}$ on the rescaled sphere is equal to the eigenvalue of the last among its $n+1$ linear functionals, i.e. to $\frac{n}{n-1}$, whereas on Gaussian space it is already equal to $2$. This show that in general, one cannot hope for a positive answer to Question \ref{ques:spectrum}.

For future reference, we record that the unscaled canonical sphere $S^n$ satisfies for all $\lambda \geq n^2$:
\begin{equation} \label{eq:sphere-concrete}
\# \set{ \lambda_k(S^n,g_{can},\vol_{g_{can}}) \leq \lambda } \geq {\lceil \sqrt{\lambda} \rceil \choose n}  \geq \brac{\frac{\sqrt{\lambda}}{n}}^n .
\end{equation}

\subsection{Weyl's asymptotic law for weighted-manifolds} \label{subsec:Weyl}

When $(M^n,g)$ is compact, then as soon as the density of $\mu$ is bounded away from $0$ and $\infty$, the classical Weyl law \cite{ChavelEigenvalues} for the eigenvalue asymptotics of the unweighted Laplacian $-\Delta_g$  applies to the weighted one $-\Delta_{g,\mu}$, and we have as $\lambda \rightarrow \infty$:
\begin{equation} \label{eq:classical-Weyl}
\# \set{\lambda_k(M^n,g,\mu) \leq \lambda} = \frac{\vol(B_2^n)}{(2\pi)^n} \vol_g(M^n) \lambda^{n/2}  ( 1 + o(1)) .
\end{equation}
Note that by Bishop's volume comparison theorem \cite{GHLBookEdition3}, $\vol_g(M^n) \leq \vol_{g_{can}^\rho}(S^n)$ for any (connected) $(M^n,g)$ with $\Ric_g \geq \rho g$ with $\rho > 0$, and so we see that Weyl's formula (\ref{eq:classical-Weyl}) confirms Conjecture \ref{conj:3} in an asymptotic sense (as $k \rightarrow \infty$):
\begin{equation} \label{eq:asymptotic-conj3}
 \frac{\lambda_k(M^n,g,\vol_g)}{\lambda_k(S^n,g^\rho_{can},\vol_{g_{can}^\rho})} \geq 1+o(1)   .
\end{equation}

When $(M^n,g)$ is non-compact, the situation is more delicate. As we have not found an explicit reference in the literature, we derive the asymptotics ourselves from the known results for Schr\"{o}dinger operators, and for simplicity, we restrict to the Euclidean case $(M^n,g) = (\Real^n,\abs{\cdot})$.

Recall that $\mu = \exp(-V(x)) dx$. Denote by $U : L^2(dx) \rightarrow L^2(\mu)$ the isometric isomorphism given by the multiplication operator $U(f) = f \exp(V/2)$. 
Conjugating $-\Delta_\mu$ by $U$, we obtain the Schr\"{o}dinger operator $H : L^2(dx) \rightarrow L^2(dx)$ given by:
\begin{align*}
H(f) & := U^{-1} (-\Delta_\mu) U (f) = -\Delta f + W f ,\\
 W & :=  U^{-1} (-\Delta_\mu) U (1) = \frac{1}{4} \abs{\nabla V}^2 -  \frac{1}{2} \Delta V ,
\end{align*}
with domain $Dom(H) = U^{-1}(Dom(\Delta_\mu))$. This well-known procedure is a form of Doob's h-transform (e.g. \cite[Section 1.15.8]{BGL-Book},\cite[Section IV]{HTransform-Reference}). 
Since $-\Delta_\mu$ and $H$ are unitarily equivalent they are both self-adjoint on their respective domains and have the same spectrum. 
Clearly $C_c^\infty(\Real^n) \subset Dom(\Delta_\mu)$ and hence $C_c^\infty(\Real^n) \subset Dom(H)$. Since it is well known (e.g. \cite{Kato-EssentialSelfAdjointnessOfSchrodinger}) that a Schr\"{o}dinger operator $H$ is essentially self-adjoint on $C_c^\infty(\Real^n)$ as soon as $W$ is in $L^2_{loc}(\Real^n)$ and bounded from below, it follows that in such a case its unique self-adjoint extension necessarily coincides with the one described above having domain $Dom(H)$. Consequently, we may apply the known Weyl formula for eigenvalue asymptotics of self-adjoint Schr\"{o}dinger operators (e.g. \cite{Gurarie-SpectralTheoryForEllipticOperators, BirmanSolomyak-Book, Shubin-SpectralTheoryBook}),
which asserts that under suitable regularity assumptions:
\begin{equation} \label{eq:Weyl}
\# \set{\lambda_k(\Real^n,\abs{\cdot},\mu) \leq \lambda} = \frac{\Theta(\lambda) }{(2\pi)^n} (1 + o(1)) ~,~ \Theta(\lambda) = \text{Vol}_{\Real^{n} \times \Real^n}(\Xi_{\lambda}) ,
\end{equation}
where $\Xi_{\lambda}$ denotes the following phase-space level set of the operator's symbol:
\[
\Xi_{\lambda} := \set{ ( x,p) \in \Real^n \times \Real^n \; ; \; \abs{p}^2 + W(x) < \lambda} .
\]
More precisely (see e.g. \cite[Theorem 6]{Gurarie-SpectralTheoryForEllipticOperators}), (\ref{eq:Weyl}) holds under the assumptions that:
\begin{enumerate}
\item $W$ is smooth and bounded below.
\item $\int_{\Real^n} \frac{dx}{(1+W_+(x))^{C}} < \infty$ for some $C > 0$. 
\item For some $0 < \alpha < \beta$ we have $\alpha \leq \lambda \frac{d}{d\lambda} \log \Theta(\lambda) \leq \beta  \;\;\; \forall \lambda > 0 $. 
\end{enumerate}
\begin{rem}
It is frequently assumed in the study of Schr\"{o}dinger operators that $W \geq 0$ in order to obtain a positive semi-definite operator, and this is also the standing assumption in \cite{Gurarie-SpectralTheoryForEllipticOperators}. However, if $W$ is only assumed bounded below, we can simply consider $H_2 = H + C$ where $C \geq 0$ is a constant so that $W_2 = W + C \geq 0$; the resulting shift in the spectrum is immaterial for the asymptotic distribution of eigenvalues, thereby justifying our slightly extended assumptions above.
The assumption that $W$ is bounded below also ensures that $H$ is essentially self-adjoint on $C_c^\infty(\Real^n)$, as explained above.
\end{rem}

Note that in typical situations (e.g. as in the next subsection):
\begin{equation} \label{eq:W-asy}
W(x) = \frac{1}{4} \abs{\nabla V(x)}^2 (1 + o(x)) \text{ as } \abs{x} \rightarrow \infty .
\end{equation}
Assuming w.l.o.g. that the minimum of $V$ is attained at the origin, it follows that if $\nabla^2 V \geq \rho Id$ with $\rho > 0$, then $\abs{\nabla V(x)} \geq \rho \abs{x}$ and hence $W(x) \geq \frac{1}{4} \rho^2 \abs{x}^2 (1 + o(x))$. In that case, (\ref{eq:Weyl}) implies that as $\lambda \rightarrow \infty$:
\[
\# \set{\lambda_k(\Real^n,\abs{\cdot},\mu) \leq \lambda} \leq  \# \set{\lambda_k(\Real^n,\abs{\cdot},\gamma^n_\rho) \leq \lambda} (1+o(1)) \; , 
\]
in asymptotic accordance with Theorem \ref{thm:spectrum}. An extension of this reasoning to the manifold setting would similarly asymptotically confirm Conjecture \conjone, but we do not pursue the details here. 

For future reference, it will be more convenient to rewrite (\ref{eq:Weyl}) as:
\begin{eqnarray}
\nonumber \Theta(\lambda) & = & \vol(B_2^n) \int_{\Real^n} (\lambda - W(x))_+^{n/2} dx \\
 \label{eq:Theta} & = & \vol(S^{n-1}) \int_0^{\sqrt{\lambda}} r^{n-1} Vol(\set{W \leq \lambda - r^2}) dr .
\end{eqnarray}

\subsection{Asymptotics for the measures $\nu^n_p$}

Let us now calculate the asymptotic distribution of eigenvalues for the product measures $\nu^n_p = c_p^n \exp(-\frac{1}{p} \sum_{i=1}^n \abs{x_i}^p) dx$, $p \in (1,\infty)$, which appear in various places in this work.
We exclude the case of the exponential measure $p=1$ since $-\Delta_{\nu^n_1}$ does not have discrete spectrum (this will also be apparent from the ensuing calculations). 
Fixing $p \in (1,\infty)$, we have:
\[
V(x) = \frac{1}{p} \sum_{i=1}^n \abs{x_i}^p - n \log c_p ,
\]
and:
\[
W(x) =  \frac{1}{4} \sum_{i=1}^n \abs{x_i}^{2p-2} - \frac{p-1}{2} \sum_{i=1}^n \abs{x_i}^{p-2} .
\]
An application of H\"{o}lder's inequality verifies that $W \geq 0$ outside the compact set $(2(p-1))^{1/p} n^{1/(2p-2)} B_{2p-2}^n$ when $p \in [2,\infty)$, and so it is clearly bounded below in that case. We now address the unboundedness of $W$ from below when $p \in (1,2)$, in tandem with the minor nuisance that $V$ and $W$ are not smooth on the coordinate hyperplanes (for any non-even $p \in (1,\infty)$). Indeed, we may 
always approximate $V$ by smooth functions $V_\eps$ so that:
\[
 \norm{V - V_\eps}_{L^\infty} \leq \eps \text{ and } \snorm{\abs{\nabla V_\eps}^2 - \abs{\nabla V}^2}_{L^\infty} \leq \eps ,
 \]
and so that in addition, when $p \in [2,\infty)$, $\norm{\Delta V_\eps - \Delta V}_{L^\infty} \leq \eps$, whereas when $p \in (1,2)$, $\Delta V_\eps$ is bounded above uniformly in $\eps>0$. The min-max principle (recalled in Section \ref{sec:contractions}) will immediately ensure that this results in a perturbation in the spectrum by a multiplicative factor of at most $\exp(2 \eps)$, which can be made arbitrarily close to $1$. The above properties of $W_\eps = \frac{1}{4} \abs{\nabla V_\eps}^2 - \frac{1}{2} \Delta V_\eps$ ensure that: 
\[
W_\eps(x) = \frac{1}{4} \sum_{i=1}^n \abs{x_i}^{2p-2} (1 + o(1)) \text{ as } \abs{x} \rightarrow \infty ,
\]
uniformly in $\eps > 0$. It is then easy to see that the function $\Theta_\eps(\lambda)$ has polynomial growth, and so the regularity and boundedness assumptions described above are satisfied, and the asymptotic distribution (\ref{eq:Weyl}) is valid, uniformly in $\eps > 0$. As in (\ref{eq:W-asy}), 
we conclude from (\ref{eq:Theta}) that:
\begin{align*}
(1+o(1)) \Theta_\eps(\lambda) & = Vol(S^{n-1}) Vol(B_{2p-2}^n) \int_0^{\sqrt{\lambda}} r^{n-1} (4 (\lambda - r^2))^{\frac{n}{2p-2}} dr \\
& = Vol(S^{n-1}) Vol(B_{2p-2}^n) 2^{\frac{n}{p-1} -1} B\brac{\frac{n}{2} , \frac{n}{2(p-1)} + 1} \lambda^{\frac{n}{2} \frac{p}{p-1}} ,
\end{align*}
where $B(x,y)$ denotes the Beta function. Plugging in the known formulae for the volume of the $\ell_{2p-2}^n$ ball and the Beta function, we finally obtain:
\begin{equation} \label{eq:WeylForGammaP}
\# \set{\lambda_k(\Real^n,\abs{\cdot},\nu^n_p) \leq \lambda} = \frac{2^{\frac{n}{p-1}} \Gamma(\frac{1}{2(p-1)} + 1)^n}{ \pi^{n/2} \Gamma(\frac{n}{2} \frac{p}{p-1} + 1) } \lambda^{\frac{n}{2} \frac{p}{p-1}} (1 + o(1)) . 
\end{equation}
Note that for $p=2$, as $\Gamma(3/2) = \sqrt{\pi}/2$, this is in precise agreement with the calculation carried out for the Gaussian case (\ref{eq:Gaussian-Theta}). Also observe that as $p\rightarrow 1$ the function on the right-hand side explodes, in accordance with the formation of essential (non-discrete) spectrum in the limiting case $p=1$. Finally observe that as $p \rightarrow \infty$, $\nu^n_p$ converges to the uniform measure on $B_\infty^n = [-1,1]^n$, and the right-hand side converges to:
\[
\frac{\lambda^{n/2}}{\pi^{n/2} \Gamma(n/2+1)} (1+o(1)) = \frac{Vol(B_2^n)}{(2\pi)^n} Vol([-1,1]^n) \lambda^{n/2} (1+o(1)) ,
\]
in precise accordance with the classical Weyl estimate (\ref{eq:classical-Weyl}).

\subsection{Cwikel--Lieb--Rozenblum inequality} \label{subsec:CLR}

Let $(M^n,g,\vol_g)$ be a connected manifold satisfying $\Ric_g \geq \rho g$ with $\rho > 0$. We have already seen in Subsection \ref{subsec:Weyl} that Weyl's law confirms Conjecture 3 in an asymptotic sense, but without delving into its proof, it is not possible to extract from it non-asymptotic estimates on the individual eigenvalues.

However, individual (loose) estimates may be obtained from the B\'erard--Gallot heat-kernel estimate mentioned in the Introduction (\cite{BerardGallot-HeatEquationComparison,Berard-LecturesOnSpectralGeometry}, \cite[Appendix]{Besson-BeurlingDenyPrinciple}):
\[
\sum_{k=1}^\infty \exp(-t \lambda_k(M^n,g,\vol_g)) \leq \sum_{k=1}^\infty \exp(-t \lambda_k(S^n,g^\rho_{can},\vol_{g^\rho_{can}})) \;\;\; \forall t > 0 ,  
\]
which confirms Conjecture 3 in an averaged sense. In particular, this may be used to prove Conjecture 3 up to a multiplicative constant $C_n > 0$ (in fact, without restricting to manifolds diffeomorphic to the sphere). 

\medskip

In this subsection, we mention yet another method for obtaining (non-sharp but quite good) explicit estimates on the individual eigenvalues, by means of the Cwikel--Lieb--Rozenblum inequality \cite{Cwikel-CLR,Lieb-CLR-Announced,Rozenblum-CLR} (see also \cite{Lieb-CLR,RozenblumSolomyak-Survey,Laptev-LiebThirringSurvey}) for the number of negative eigenvalues of the Schr\"{o}dinger operator $-\Delta + W$, established independently and by different means by these three authors. Lieb's approach \cite{Lieb-CLR} relied on the ultracontractivity of the associated semi-group, whereas Li and Yau provided in \cite{LiYau-PolyaAndCLR} yet another proof based on the Sobolev inequality; these two assumptions were shown to be equivalent by Varopoulos \cite{Varopoulos-EquivalenceOfSobolevAndUltracontractivity} (see also \cite{BGL-Book}). The CLR inequality has since been generalized to a very abstract setting, and we employ the following version by Levin and Solomyak from \cite{LevinSolomyak-CLRForSemiGroups}:

\begin{thm*}[Generalized CLR inequality, Levin--Solomyak]
Let $(\Omega,\mu)$ denote a $\sigma$-finite measure space. 
Let $A$ denote a self-adjoint operator on $L^2(\mu)$ with associated non-negative closed Dirichlet form $Q[f] = \int (A f) f d\mu$ and dense domain $Dom(Q)$. Assume that $A$ generates a symmetric positivity preserving semi-group $\exp(-t A)$, and that for some $q > 2$:
\[
\norm{f}_{L^q(\mu)}^2 \leq Q[f]  \;\;\; \forall f \in Dom(Q) .
\]
If:
\[
0 \leq W \in L^p(\mu) \;\;\;, \;\;\; p := \frac{q}{q-2}  \;,
\]
then the negative spectrum of $A-W$ is discrete and:
\[
\# \set{\lambda_k(A - W) \leq 0}  \leq e^p \int W^p d\mu .
\]
\end{thm*}

In order to apply this result to our manifold (for $n \geq 3$), we employ the sharp Sobolev inequality on $\CD(\rho,N)$ weighted manifolds with finite $N \in [n,\infty)$ and $\rho > 0$ (\cite[Theorem 6.8.3]{BGL-Book}). In our setting ($N=n$), it states that:
\begin{equation} \label{eq:Sobolev}
\norm{f}_{L^q(\mu)}^2 \leq \norm{f}_{L^2(\mu)}^2 + C_{n,\rho}\int_M \abs{\nabla f}^2 d\mu =: Q[f] \;\;\; \forall f \in C^\infty(M)   ,
\end{equation}
where $\mu = \widetilde{\vol_g}$, $q := \frac{2n}{n-2}$ and $C_{n,\rho} :=\frac{4 (n-1)}{n (n-2) \rho}$. 

\begin{cor}
Let $(M^n,g)$ denote a connected $n$-dimensional Riemannian manifold with $\Ric_g \geq \rho g$, $\rho > 0$ and $n \geq 3$. Then for all $\lambda > 0$:
\[
\# \set{ \lambda_k(M^n,g,\vol_g) \leq \lambda } \leq e^{\frac{n}{2}} (C_{n,\rho} \lambda + 1)^{\frac{n}{2}} \;\;,\;\;  C_{n,\rho} :=\frac{4 (n-1)}{n (n-2) \rho} .
\]
\end{cor}
\begin{proof}
Apply the generalized CLR inequality to the self-adjoint operator $A = -C_{n,\rho} \Delta_g + Id$ on $L^2(\mu)$, $\mu = \widetilde{\vol_g}$, with $W \equiv  C_{n,\rho} \lambda + 1$. The associated Dirichlet form $Q$ satisfies (\ref{eq:Sobolev}), and so the assertion follows since $p = \frac{n}{2}$ and:
\[
\# \set{ \lambda_k (-\Delta_g) \leq \lambda} = \#  \set{ \lambda_k(A) \leq C_{n,\rho} \lambda + 1} = \# \set{\lambda_k(A - W) \leq 0}  .
\]
\end{proof}

Comparing the latter estimate for $\rho=n-1$ to the eigenvalue distribution (\ref{eq:sphere-concrete}) on the canonical $n$-sphere (having $\Ric_{g_{can}} = (n-1) g_{can}$), we see that for all $\lambda \geq n^2$:
\begin{align*}
& \# \set{ \lambda_k(M^n,g,\vol_g) \leq \lambda } \leq \brac{1-\frac{2}{n}}^{-\frac{n}{2}} (5 e)^{\frac{n}{2}} \frac{\lambda^{n/2}}{n^n} \; , \\
& \# \set{ \lambda_k(S^n,g_{can},\vol_{g_{can}}) \leq \lambda } \geq \frac{\lambda^{n/2}}{n^n} .
\end{align*}
Since our choice of $\rho=n-1$ only influences the scaling of both estimates, we obtain the following:
\begin{cor}
Let $(M^n,g)$ denote a connected $n$-dimensional Riemannian manifold with $\Ric_g \geq \rho g$, $\rho > 0$ and $n \geq 3$. Then:
\[
\lambda_k(M^n,g,\vol_g)  \geq  \brac{1-\frac{2}{n}} \frac{1}{5 e} \lambda_k(S^n,g^\rho_{can},\vol_{g_{can}^\rho}) \;\;\; \forall k \geq 6 (5e)^{\frac{n}{2}} .
\]
In other words, Conjecture 3 holds true (in fact without assuming that $M^n$ is diffeomorphic to $S^n$) for exponentially large (in $n$) eigenvalues, up to a multiplicative numeric constant (independent of $n$). 
\end{cor}
The latter corollary provides a concrete non-asymptotic version of (\ref{eq:asymptotic-conj3}). The fact that it is dimension independent as soon as $k$ is exponentially large is rather satisfying.

\section{Contractions and Spectrum} \label{sec:contractions}

\subsection{On metric-measure spaces} 

A triplet $(\Omega,d,\mu)$ is called a metric-measure space if $(\Omega,d)$ is a separable metric space and $\mu$ is a locally-finite Borel measure on $(\Omega,d)$.
Throughout this subsection, a local property is one which holds on some open neighborhood of any given point. 
Given a weighted-manifold $(M,g,\mu)$, we will always equip it with its induced geodesic distance $d$, so that $(M,g,\mu)$ constitutes a metric-measure space. 
Finally, note that $T : (\Omega_1,d_1,\mu_1) \rightarrow (\Omega_2,d_2,\mu_2)$ pushes forward $\mu_1$ onto $\mu_2$ iff:
\[
\int f (T x) d\mu_1(x) = \int f(y) d\mu_2(y) ,
\]
for any integrable function $f$ on $(\Omega_2,d_2,\mu_2)$. 

\begin{dfn*}
Let $\F = \F(\Omega,d)$ denote the class of locally Lipschitz functions. Given $f \in \F$, its local Lipschitz constant is defined as the following (Borel measurable) function on $(\Omega,d)$:
\[
\abs{\nabla f}(x) := \limsup_{y \rightarrow x} \frac{\abs{f(y) - f(x)}}{d(y,x)} ~;
\]
(and we define it as 0 if x is an isolated point - see \cite{BobkovHoudreMemoirs} for more details). Note that on $(M,g)$, $\abs{\nabla f}$ clearly coincides with the Riemannian length of the gradient for any $f \in C^1(M)$. 
\end{dfn*}

\begin{dfn*}
Given a Borel measurable function $f$ on $(\Omega,d)$, set:
\[
\norm{f}^2_{\dot{W}^{1,2}_{Lip}(\mu)} = \norm{f}^2_{\dot{W}^{1,2}_{Lip}(\Omega,d,\mu)} := \inf \set{ \int \abs{\nabla g}^2 d\mu \; ; \; g \in \F(\Omega,d) ~,~ g=f \; \text{$\mu$-a.e.} } .
\]
Note that $f \mapsto \norm{f}_{\dot{W}^{1,2}_{Lip}(\mu)}$ is homogeneous, satisfies the triangle inequality, and is invariant to changing $f$ $\mu$-\text{a.e.}; it may be thought of as a homogeneous metric Sobolev semi-norm. 
In addition, if $\mu$ has full support, then clearly $\norm{f}^2_{\dot{W}^{1,2}_{Lip}} = \int \abs{\nabla f}^2 d\mu$ for all $f \in \F(\Omega,d)$, but this may not be the case in general. 
\end{dfn*}

\begin{dfn*}
The metric Sobolev space $W^{1,2}_{Lip}(\mu) = W^{1,2}_{Lip}(\Omega,d,\mu)$ is defined as the subspace of $L^2(\mu)$ consisting of elements $f$ with $\norm{f}_{\dot{W}^{1,2}_{Lip}(\mu)} < \infty$. By the above comments, it is a linear subspace, and inherits the scalar-product structure of $L^2(\mu)$. 
\end{dfn*}

\begin{dfn*}
The $k$-th metric eigenvalue of $(\Omega,d,\mu)$ is defined as:
\begin{align}
\label{eq:Llambda} \Llambda_k = \Llambda_k(\Omega,d,\mu) & := \sup_{\footnotesize \begin{array}{c} E \subset W^{1,2}_{Lip}(\mu) \\\dim E = k-1 \end{array}} \inf_{\footnotesize \begin{array}{c} f \in W^{1,2}_{Lip}(\mu) \cap E^\perp \\ f \neq 0 \end{array}} \frac{\norm{f}^2_{\dot{W}^{1,2}_{Lip}(\mu)}}{\norm{f}^2_{L^2(\mu)}} \\
\nonumber & =  \inf_{\footnotesize \begin{array}{c} F \subset W^{1,2}_{Lip}(\mu) \\\dim F = k \end{array}} \sup_{\footnotesize \begin{array}{c} f  \in F \\ f \neq 0 \end{array}} \frac{\norm{f}^2_{\dot{W}^{1,2}_{Lip}(\mu)}}{\norm{f}^2_{L^2(\mu)}}  ,
\end{align}
where $E,F,E^\perp$ range over linear subspaces of $W^{1,2}_{Lip}(\mu)$, and $0$ is understood as the zero element of $W^{1,2}_{Lip}(\mu)$.
The equivalence between the latter two expressions follows from elementary linear algebra and is standard. 
Clearly $ 0 \leq \Llambda_1 \leq \Llambda_2 \leq \ldots$ is a non-decreasing sequence. 
\end{dfn*}

\begin{prop}[Metric Contraction Principle] \label{prop:metric-contract}
Let $T : (\Omega_1,d_1,\mu_1) \rightarrow (\Omega_2,d_2,\mu_2)$ denote a contraction pushing forward $\mu_1$ onto $\mu_2$ up to a finite constant $c \in (0,\infty)$. Then:
\[
\Llambda_k(\Omega_2,d_2,\mu_2) \geq \Llambda_k(\Omega_1,d_1,\mu_1) \;\;\; \forall k \geq 1 .
\]
\end{prop}
\noindent
Note that no injectivity is assumed above, nor do we assume that $\mu_i$ have full supports. 
\begin{proof}
Given $g \in  \F(\Omega_2,d_2)$, the contraction property clearly implies that $g \circ T \in \F(\Omega_1,d_1)$ and:
\[
\abs{\nabla (g \circ T)}(x) \leq \abs{\nabla g}( T(x)) \;\;\; \forall x \in \Omega_1 . 
\]
Consequently:
\[
\int \abs{\nabla (g \circ T)}^2(x) d\mu_1(x) \leq \int \abs{\nabla g}^2(T x) d\mu_1(x) = c \int \abs{\nabla g}^2(y) d\mu_2(y) .
\]
Taking infimum over all $g \in \F(\Omega_2,d_2)$ so that $g =f$ $\mu_2$-a.e. (and hence $g \circ T = f \circ T$ $\mu_1$-a.e.), we obtain for a given Borel measurable $f$ on $(\Omega_2,d_2)$:
\[
\norm{f \circ T}^2_{\dot{W}^{1,2}_{Lip}(\Omega_1,d_1,\mu_1)} \leq \inf_{g \text{ as above}} \int \abs{\nabla (g \circ T)}^2(x) d\mu_1(x) \leq c \norm{f }^2_{\dot{W}^{1,2}_{Lip}(\Omega_2,d_2,\mu_2)} .
\]
In addition, obviously:
\[
\norm{f \circ T}^2_{L^2(\mu_1)} =  c \norm{f}^2_{L^2(\mu_2)} .
\]

Now given a linear subspace $F_2$ of $W^{1,2}_{Lip}(\Omega_2,d_2,\mu_2)$ of dimension $k$, consider its pull-back:
\[
T^* F_2 := \set{ f \circ T \; ; \; f \in F_2 } .
\]
By the above reasoning, $T^* F_2$ is a linear subspace of $W^{1,2}_{Lip}(\Omega_1,d_1,\mu_1)$, and we have for all $0 \neq f \in F_2$:
\[
\frac{\norm{f \circ T}^2_{\dot{W}^{1,2}_{Lip}(\Omega_1,d_1,\mu_1)}}{\norm{f\circ T}^2_{L^2(\mu_1)}} \leq  
\frac{\norm{f }^2_{\dot{W}^{1,2}_{Lip}(\Omega_2,d_2,\mu_2)}}{\norm{f}^2_{L^2(\mu_2)}} .
\]
It remains to note that $F_2$ and $T^* F_2$ are in fact linearly isomorphic, and hence have the same dimension. Indeed, $f_1,\ldots,f_k$ is a basis of $F_2$ if and only if $f_1 \circ T , \ldots, f_k \circ T$ is a basis of $T^* F_2$, since $\sum_{i=1}^k c_i f_i = 0$ in $L^2(\mu_2)$ if and only if $\sum_{i=1}^k c_i f_i \circ T = 0$ in $L^2(\mu_1)$ (as $T$ pushes forward $\mu_1$ onto $\mu_2$ up to a finite constant). Note that the latter property would be false without identifying functions in $W^{1,2}_{Lip}(\mu)$ which coincide $\mu$-a.e..

Finally, taking supremum over all $0 \neq f \in F_2$, followed by an infimum over all $F_2$ as above, we obtain: 
\begin{align*}
\Llambda_k(\Omega_2,d_2,\mu_2) & = \inf_{\dim F_2 = k} \sup_{0 \neq f \in F_2} \frac{\norm{f }^2_{\dot{W}^{1,2}_{Lip}(\Omega_2,d_2,\mu_2)}}{\norm{f}^2_{L^2(\mu_2)}} \\
& \geq \inf_{\dim F_2 = k} \sup_{0 \neq h \in T^* F_2} \frac{\norm{h }^2_{\dot{W}^{1,2}_{Lip}(\Omega_1,d_1,\mu_1)}}{\norm{h}^2_{L^2(\mu_1)}} \\
& \geq \inf_{\dim F_1 = k} \sup_{0 \neq h \in F_1} \frac{\norm{h }^2_{\dot{W}^{1,2}_{Lip}(\Omega_1,d_1,\mu_1)}}{\norm{h}^2_{L^2(\mu_1)}}  = \Llambda_k(\Omega_1,d_1,\mu_1) ,
\end{align*}
where $F_1$ ranges over all $k$-dimensional linear subspaces of $W^{1,2}_{Lip}(\Omega_1,d_1,\mu_1)$. Note that the last inequality above may be strict if $T$ is not injective, or even if it is injective but $T^{-1}$ is not Lipschitz. This concludes the proof. 
\end{proof}

\subsection{On complete weighted-manifolds} \label{subsec:manifold-contraction}

Let $(M^n,g,\mu)$ denote a weighted (complete) Riemannian manifold without boundary, and let $\Delta_{g,\mu}$ denote the associated weighted Laplacian. Recall (e.g. \cite[Proposition 3.2.1]{BGL-Book}) 
that the completeness ensures that the linear operator $-\Delta_{g,\mu}$ on $L^2(\mu)$ is positive semi-definite and essentially self-adjoint when acting on the dense domain of compactly supported smooth functions $C^\infty_c(M)$. Consequently, its graph-closure is its unique self-adjoint extension, and we continue to denote the resulting positive semi-definite self-adjoint operator by $-\Delta_{g,\mu}$, with corresponding domain $Dom(\Delta_{g,\mu})$. 

By the spectral theory of self-adjoint operators (e.g. \cite{Davies-SpectralTheoryBook}),
the spectrum $\sigma(-\Delta_{g,\mu})$ is a subset of $[0,\infty)$. Furthermore, there exists a spectral decomposition of identity, i.e. a family $\set{E_\lambda}_{\lambda \in [0,\infty)}$ of orthogonal projections on $L^2(\mu)$ so that for all $f \in Dom(\Delta_{g,\mu})$:
\[
-\Delta_{g,\mu} f  = \int_0^\infty \lambda \; dE_\lambda f \;\; ;
\]
we refer to \cite{Davies-SpectralTheoryBook} for further details. If the spectrum is discrete, i.e. $dim(Im(E_\lambda)) < \infty$ for all $\lambda \in [0,\infty)$, then it is composed of isolated eigenvalues of finite multiplicity; we denote these by $\lambda_k = \lambda_k(M^n,g,\mu)$, and arrange them in non-decreasing order (repeated by multiplicity) 
$0 \leq \lambda_1 \leq \lambda_2 \leq \ldots $. More generally, we define:
\[
\lambda_k(M,g,\mu) := \sup \set{ \lambda \in \Real \; ; \; dim(Im(E_\lambda)) < k } .
\]
Note that the spectrum is discrete if and only if $k \mapsto \lambda_k$ increases to infinity. 

\medskip
The following is essentially standard; for completeness, we provide a proof. 

\begin{prop} \label{prop:same-spectrum}
If $(M,g,\mu)$ is a complete weighted-manifold without boundary, then:
\[
 \lambda_k(M,g,\mu) = \Llambda_k(M,g,\mu)\;\;\; \forall k \geq 1 .
\]
\end{prop}
\begin{proof}
The essential self-adjointness of $\Delta_{g,\mu}$ on $C_c^\infty(M)$ implies that the latter space is dense in $\D = Dom(\Delta_{g,\mu})$ in the graph norm $\norm{f}_{\Gamma}^2 = \norm{f}_{L^2(\mu)}^2 + \norm{\Delta_{g,\mu} f}_{L^2(\mu)}^2$. Consequently, by the Courant--Fischer min-max theorem \cite{Davies-SpectralTheoryBook}, density, and integration by parts:
\begin{align*}
\lambda_k(M,g,\mu) & = \inf_{\footnotesize \begin{array}{c} F \subset \D \\ \dim F = k \end{array}} \sup_{\footnotesize \begin{array}{c} f  \in F \\ f \neq 0 \end{array}} \frac{\scalar{-\Delta_{g,\mu} f , f}}{\scalar{f,f}} 
=  \inf_{\footnotesize \begin{array}{c} F \subset C_c^\infty(M) \\ \dim F = k \end{array}} \sup_{\footnotesize \begin{array}{c} f  \in F \\ f \neq 0 \end{array}} \frac{\scalar{-\Delta_{g,\mu} f , f}}{\scalar{f,f}} \\
& = \inf_{\footnotesize \begin{array}{c} F \subset C_c^\infty(M) \\ \dim F = k \end{array}} \sup_{\footnotesize \begin{array}{c} f  \in F \\ f \neq 0 \end{array}}  \frac{\int \abs{\nabla f}^2 d\mu}{\int \abs{f}^2 d\mu} = \inf_{\footnotesize \begin{array}{c} F \subset W^{1,2}(M,g,\mu) \\ \dim F = k \end{array}} \sup_{\footnotesize \begin{array}{c} f  \in F \\ f \neq 0 \end{array}}  \frac{\int \abs{\nabla f}^2 d\mu}{\int \abs{f}^2 d\mu} .
\end{align*}
Here $W^{1,2}(M,g,\mu)$ denotes the completion of $C^\infty_c(M)$ in the $W^{1,2}(M,g,\mu)$ norm:
\[
\norm{f}^2_{W^{1,2}(M,g,\mu)} := \int_M\abs{f}^2 d\mu + \int_M \abs{\nabla f}^2 d\mu .
\]
On the other hand, since $\mu$ has full support on $M$, we have:
\[
\Llambda_k(M,g,\mu) = \inf_{\footnotesize \begin{array}{c} F \subset W^{1,2}_{Lip}(M,g,\mu) \\ \dim F = k \end{array}} \sup_{\footnotesize \begin{array}{c} f  \in F \\ f \neq 0 \end{array}}  \frac{\int \abs{\nabla f}^2 d\mu}{\int \abs{f}^2 d\mu} .
\]
Clearly $C_c^\infty(M) \subset W^{1,2}_{Lip}(M,g,\mu)$, 
and so to establish the assertion, it remains to show that $W^{1,2}_{Lip}(M,g,\mu) \subset W^{1,2}(M,g,\mu)$, or equivalently, that $C_c^\infty(M)$ is dense in $W^{1,2}_{Lip}(M,g,\mu)$ (in the $W^{1,2}(M,g,\mu)$ norm). But these two equivalent statements are well-known and standard. The first may be seen by the known identification of $W^{1,2}(M,g,\mu)$ as the space of weakly differentiable functions with finite $W^{1,2}(M,g,\mu)$ norm \cite{Taylor-PDE-II-Book}, and the fact that locally Lipschitz functions are differentiable almost-everywhere by Rademacher's theorem. Alternatively, the second statement is easily seen as follows: the completeness of $(M,g)$ implies the existence of a family of functions $\varphi_m \in C_c^\infty(M)$  with $\norm{\abs{\nabla \varphi_m}}_{L^\infty}\leq 1/m$ so that $\varphi_m \rightarrow_{m \rightarrow \infty} 1$ pointwise on $M$ \cite{BGL-Book}; consequently, given $f \in W^{1,2}_{Lip}(M,g,\mu)$, $f \varphi_m \rightarrow f$ in $W^{1,2}(M,g,\mu)$,
and so it is enough to establish that any compactly supported function in $W^{1,2}_{Lip}(M,g,\mu)$ may be approximated by functions in $C_c^\infty(M)$; but the latter is standard, using a partition of unity and a mollification argument. This concludes the proof. 
\end{proof}

Propositions \ref{prop:metric-contract} and \ref{prop:same-spectrum} immediately yield:
\begin{cor}
Let $T : (M_1,g_1,\mu_1) \rightarrow (M_2,g_2,\mu_2)$ denote a contraction between two weighted-manifolds pushing-forward $\mu_1$ onto $\mu_2$ up to a finite constant. Then:
\[
\lambda_k(M_2,g_2,\mu_2) \geq \lambda_k(M_1,g_1,\mu_1) \;\;\; \forall k \geq 1. 
\]
In particular, if $\Delta_{g_1,\mu_1}$ has discrete spectrum, then so does $\Delta_{g_2,\mu_2}$.
\end{cor}

\subsection{Lipschitz maps}

The above results trivially extend to globally Lipschitz maps $T : (\Omega_1,d_1) \rightarrow (\Omega_2,d_2)$. Indeed, if the map $T$ is $L$-Lipschitz, then by scaling either of the metrics by a factor of $L$, the map $T$ becomes a contraction. Since the (metric) eigenvalues clearly scale by a factor of $1/L^2$, we immediately obtain that if in addition $T$ pushes forward $\mu_1$ onto $\mu_2$ up to a finite constant, then:
\[
\Llambda_k(\Omega_2,d_2,\mu_2) \geq \frac{1}{L^2} \Llambda_k(\Omega_1,d_1,\mu_1) \;\;\; \forall k \geq 1 ,
\]
yielding the Contraction Principle (Theorem \ref{thm:intro-Lipschitz}) from the Introduction. 

\subsection{Extensions to compact weighted-manifolds with boundary} \label{subsec:boundary}

While for simplicity we generally avoid manifolds with boundary in this work, we comment here that it is a standard exercise to extend the previous results to eigenvalues of the weighted Laplacian with either Dirichlet or Neumann boundary conditions on compact weighted-manifolds with boundary.

To treat Dirichlet boundary conditions from the purely metric-measure space view-point, we proceed as follows. Given a Borel subset $\Lambda \subset \Omega$, denote by $\F_\Lambda(\Omega,d)$ the subspace of functions in $\F(\Omega,d)$ which vanish on $\Lambda$. Replacing $\F(\Omega,d)$ by $\F_{\Lambda}(\Omega,d)$ in the definition of $\snorm{\cdot}_{\dot{W}^{1,2}_{Lip}(\Omega,d,\mu)}$, we obtain $\snorm{\cdot}_{\dot{W}^{1,2}_{Lip,\Lambda}(\Omega,d,\mu)}$. Similarly, $W^{1,2}_{Lip,\Lambda}(\Omega,d,\mu)$ is defined as the linear subspace of $L^2(\mu)$ consisting of elements $f$ with $\snorm{f}_{\dot{W}^{1,2}_{Lip,\Lambda}(\Omega,d,\mu)}<\infty$. 
 The metric eigenvalues $\lambda_k^{Lip,\Lambda}(\Omega,d,\mu)$ are defined as in (\ref{eq:Llambda}) with $W^{1,2}_{Lip,\Lambda}(\Omega,d,\mu)$ and $\snorm{\cdot}_{\dot{W}^{1,2}_{Lip,\Lambda}(\Omega,d,\mu)}$ replacing $W^{1,2}_{Lip}(\Omega,d,\mu)$ and $\snorm{\cdot}_{\dot{W}^{1,2}_{Lip}(\Omega,d,\mu)}$, respectively. Repeating verbatim the proof of Proposition \ref{prop:metric-contract}, we have:

\begin{prop}[Metric Contraction Principle with Dirichlet Conditions]
If $T : (\Omega_1,d_1,\mu_1) \rightarrow (\Omega_2,d_2,\mu_2)$ is a contraction pushing forward $\mu_1$ onto $\mu_2$ up to a finite constant, then for every Borel subset $\Lambda_2 \subset \Omega_2$, we have:
\[
\lambda_k^{Lip,\Lambda_2}(\Omega_2,d_2,\mu_2) \geq \lambda_k^{Lip,T^{-1}(\Lambda_2)}(\Omega_1,d_1,\mu_1)  \;\;\; \forall k \geq 1 .
\]
\end{prop}

Returning to the weighted-manifold setting, it is well-known that when $M$ is compact with smooth boundary,  
$-\Delta_{g,\mu}$ is positive semi-definite and essentially self-adjoint on both $C_0^\infty(M)$ and $C_\nu^\infty(M)$, the spaces of smooth functions $f$ with $f|_{\partial M} = 0$ and $f_\nu|_{\partial M} = 0$, respectively, where $f_\nu$ denotes the partial derivative in the direction orthogonal to the boundary. Denote by $-\Delta_{g,\mu}^0$ and $-\Delta_{g,\mu}^\nu$ the corresponding (Dirichlet and Neumann, respectively) self-adjoint extensions, and by $\lambda_k^0(M,g,\mu)$ and $\lambda_k(M,g,\mu)$ their corresponding eigenvalues.
Repeating the proof of Proposition \ref{prop:same-spectrum}, we obtain:

\begin{prop} \label{prop:same-eigen-boundary}
If $(M,g,\mu)$ is a compact weighted-manifold with smooth boundary, then:
\[
 \lambda_k^0(M,g,\mu) = \lambda^{Lip,\partial M}_k(M,g,\mu)\;\;\; \forall k \geq 1 \; ,
\]
and:
\[
\lambda_k(M,g,\mu) = \lambda^{Lip}_k(M,g,\mu)\;\;\; \forall k \geq 1 \; . 
\]
\end{prop}
\noindent
Indeed, the only difference with the proof of Proposition \ref{prop:same-spectrum} is in justifying that:
\begin{equation} \label{eq:bc1}
C_0^\infty(M) \subset W^{1,2}_{Lip,\partial M}(M,g,\mu) \subset \overline{C_0^\infty(M)} ,
\end{equation}
and that:
\begin{equation} \label{eq:bc2}
C_\nu^\infty(M) \subset W^{1,2}_{Lip}(M,g,\mu) \subset \overline{C_\nu^\infty(M)} ,
\end{equation}
where the closures are taken in $W^{1,2}(M,g,\mu)$. 
As before, the first inclusions in (\ref{eq:bc1}) and (\ref{eq:bc2}) are immediate, and the second ones follow from well-known arguments \cite[Chapter 8]{Taylor-PDE-II-Book}. In fact, $\overline{C_0^\infty(M)}$ is by definition the well-known space $W^{1,2}_0(M,g,\mu)$, and $\overline{C_\nu^\infty(M)} = W^{1,2}(M,g,\mu)$. 

\medskip

Putting everything together, we obtain:
\begin{thm}[Contraction Principle With Boundary] \label{thm:contract-boundary}
Let $T : (M_1,g_1,\mu_1) \rightarrow (M_2,g_2,\mu_2)$ denote an $L$-Lipschitz map between two complete weighted-manifolds pushing-forward $\mu_1$ onto $\mu_2$ up to a finite constant. If $\partial M_i \neq \emptyset$, assume that $M_i$ is compact with smooth boundary. Then the Neumann eigenvalues satisfy:
\[
\lambda_k(M_2,g_2,\mu_2) \geq \frac{1}{L^2} \lambda_k(M_1,g_1,\mu_1) \;\;\; \forall k \geq 1. 
\]
If in addition $T$ maps $\partial M_1$ onto $\partial M_2$, then the Dirichlet eigenvalues satisfy:
\[
\lambda^0_k(M_2,g_2,\mu_2) \geq \frac{1}{L^2} \lambda^0_k(M_1,g_1,\mu_1) \;\;\; \forall k \geq 1. 
\]
\end{thm}

\section{Known Contractions} \label{sec:known-contractions}

We have already formulated Caffarelli's Contraction Theorem \cite{CaffarelliContraction} in the Introduction: the optimal-transport map pushing-forward $\gamma^n_\rho$ onto any probability measure $\mu = \exp(-V(x)) dx$ satisfying $\nabla^2 V \geq \rho Id$, contracts the Euclidean metric. In this section, we describe several additional known contractions and Lipschitz maps between notable weighted-manifolds. 

\subsection{Generalized Caffarelli Contraction Theorems}

An extension of Caffarelli's theorem was obtained by Kim and the author  in \cite{KimEMilmanGeneralizedCaffarelli} - we refer the reader to \cite[Theorem 1.1] {KimEMilmanGeneralizedCaffarelli}
for the most general formulation. To avoid extraneous generality, we only state the simplest case: under the assumptions of Theorem \ref{thm:KM}, there exists a contraction pushing forward the source onto the target measure. Together with the Contraction Principle, this yields Theorem \ref{thm:KM}. 

A different extension of Caffarelli's theorem was obtained by Kolesnikov in \cite[Theorem 2.2]{KolesnikovContractionSurvey} - he noted that Caffarelli's proof in fact yields a contraction between any probability measures $\mu = \exp(-V(x)) dx$ and $\nu = \exp(-W(x)) dx$ as soon as $\nabla^2 V \leq \rho Id \leq \nabla^2 W$. 
Applying this for $\nu = \gamma^{n}_\rho$ yields a contraction pushing forward $\mu$ onto $\gamma^n_\rho$, yielding Theorem \ref{thm:Kolesnikov} by the Contraction Principle. See also \cite{Valdimarsson-GenBL} for additional generalizations. 

\subsection{Unit-ball of $\ell_p^n$}

It was shown by Lata{\l}a and Wojtaszczyk in \cite[Proposition 5.21]{LatalaJacobInfConvolution} that when $p \in [2,\infty]$, there exists a globally Lipschitz map $T$ (with respect to the Euclidean metric) pushing forward the  standard Gaussian measure $\gamma^n$ onto $\lambdanu_{\tilde{B}_p^n}$, the uniform measure on the unit-ball of $\ell_p^n$ rescaled to have unit volume. This map was obtained as the composition of the product map pushing forward $\gamma^n$ onto the product measure $\nu^n_p  = c_p^n \exp(-\frac{1}{p} \sum_{i=1}^n \abs{x_i}^p) dx$, with the radial map pushing forward $\nu^n_p$ onto $\lambdanu_{\tilde{B}_p^n}$. Its Lipschitz constant was shown to be bounded above by $14 \sqrt{2}$, thereby yielding Theorem \ref{thm:Latala} on the Neumann eigenvalues of $\tilde{B}_p^n$ by the Contraction Principle With Boundary. Note that Theorem \ref{thm:Latala} misses the correct order for the asymptotics of the eigenvalues of $\tilde{B}_p^n$, which is governed by Weyl's law (\ref{eq:classical-Weyl}) and not the Gaussian behaviour (\ref{eq:Gaussian-Theta}).

We remark here that while $\tilde{B}_p^n$ does not have smooth boundary, as formally required by Theorem \ref{thm:contract-boundary}, it is nevertheless still Lipschitz, and the analogue of Proposition \ref{prop:same-eigen-boundary} still holds with the Neumann Laplacian interpreted as the self-adjoint operator associated with the quadratic form $\int_{\tilde{B}_p^n} \abs{\nabla f}^2 d \lambdanu_{\tilde{B}_p^n}$ for $f \in W^{1,2}(\tilde{B}_p^n)$, see \cite{Davies-SpectralTheoryBook} (cf.  \cite{BurenkovDavies-SpectralStabilityForNeumannLaplacian, BurenkovLamberti-SpectralStabilityForGeneralOperatorsWithNeumannConds} regarding additional stability results for the spectrum of the Neumann Laplacian on non-smooth domains). 

\subsection{One-Dimensional Measures}

Theorem \ref{thm:Bobkov} states that the natural partial ordering on the (one-sided, flat) isoperimetric profiles of one-dimensional probability measures induces an ordering on their entire spectrum. For the proof, we adapt an observation due to Bobkov and Houdr\'e \cite{BobkovHoudre}, regarding the one-dimensional optimal-transport (or simply monotone) map $T$ between two probability measures $\mu_1,\mu_2$ on $\Real$. Let $\mu_i = f_i(t) dt$ (having full support for simplicity), denote $F_i(x) = \mu_i((-\infty,x])$, and recall that $\I^\flat_{\mu_i} := f_i \circ F_i^{-1} : [0,1] \rightarrow \Real_+$ denotes the one-sided flat isoperimetric profile. The increasing monotone map $T$ pushing forward $\mu_1$ onto $\mu_2$ is given by:
\[
v_x := F_2(T(x)) = \int_{-\infty}^{T(x)} f_2(t) dt = \int_{-\infty}^x f_1(t) dt = F_1(x) \;\;\; \forall x \in \Real . 
\]
Differentiating, we obtain:
\[
T'(x) \I^\flat_{\mu_2}(v_x)  = T'(x) f_2(T(x)) = f_1(x) = \I^\flat_{\mu_1}(v_x)  \;\;\; \forall x \in \Real . 
\]
It follows that if:
\begin{equation} \label{eq:best-Lip}
\I^\flat_{\mu_2}(v) \geq \frac{1}{L} \I^\flat_{\mu_1}(v) \;\;\; \forall v \in [0,1] ,
\end{equation}
then $0 \leq T' \leq L$, i.e. $T$ is $L$-Lipschitz (in fact, we see that the Lipschitz constant of $T$ is precisely the best possible $L$ in (\ref{eq:best-Lip})). Together with the Contraction Principle, Theorem \ref{thm:Bobkov} follows. 

Note that it is of course also possible to compare $\I^\flat_{\mu_2}(v)$ with $\I^\flat_{\mu_1}(1-v)$ in the above argument, by considering the decreasing monotone map in place of the increasing one. When one of the measures $\mu_i$ is symmetric about a point, this makes no difference. 

\subsection{Riemannian submersions} \label{subsec:submersions}

Recall that a surjective and smooth map $T : (M^{n_1}_1,g_1) \rightarrow (M^{n_2}_2,g_2)$ is called a submersion if $n_1 \geq n_2$ and at every point $x \in M_1$, the differential $d_x T$ is of maximal rank $n_2$.  A submersion is called Riemannian if $d_x T$ is an isometry on the orthogonal complement to its kernel. It immediately follows that a Riemannian submersion is a contracting map, yielding by the Contraction Principle the first part of Theorem \ref{thm:submersion}. Note that for this part, it is actually not necessary to assume that $T$ is a genuine submersion: it need not be surjective, and the differential need only be an isometry between the orthogonal complement to its kernel and its image, regardless of its rank (which may vary from point to point). 

As for the second part: $T^{-1}(y)$, the fiber over $y \in M_2$, is a smooth submanifold of $M_1$ of dimension $n_1 - n_2$ by the implicit function theorem. The fibers are called minimal if they are minimal submanifolds, i.e. their mean-curvature vector vanishes identically. It is known (e.g. \cite[Lemma 3.1]{Bordoni-SpectralSurvey}, \cite[Lemma 5.1]{Bordoni-SpectralEstimatesViaKato}) that when  $M_1$ (and thus $M_2$) are connected and all the submersion's fibers are minimal, then they are all diffeomorphic and must have the same induced Riemannian volume in $M_1$; the assumption that the fibers are compact ensures that this volume $c \in (0,\infty)$ is finite. Consequently, $T$ pushes forward the Riemannian volume measures $\vol_{g_1}$ onto $\vol_{g_2}$ up to the finite constant $c$, yielding the second part of Theorem \ref{thm:submersion}; the statement about Riemannian coverings follows immediately. This recovers an observation known to Bordoni \cite[Section 3]{Bordoni-HeatEstimatesForSubmersions}, who noted (at least when the manifolds are compact) that whenever the fibers are minimal, any eigenfunction on $M_2$ can be lifted up to an eigenfunction on $M_1$ with the same eigenvalue, immediately yielding comparison of the entire spectrum. Clearly, this lifting property is immediate for Riemannian coverings.

We remark for completeness that under the assumption that the fibers are totally geodesic, Besson and Bordoni \cite{BessonBordoni-SpectrumOfSubmersionsWithTotallyGeodesicFibers} (see also B\'erard--Gallot \cite{BerardGallot-HeatEquationComparison}) provided a complete description of the relation between the spectra of $(M_i,g_i,\vol_{g_i})$. When the fibers are only assumed minimal, Bordoni showed in \cite{Bordoni-SpectralEstimatesViaKato} an \emph{upper bound} on the eigenvalues of the target manifold as a function of the eigenvalues of the source one (in the opposite direction to the one considered in this work); in \cite{Bordoni-HeatEstimatesForSubmersions}, Bordoni employed a generalized Beurling-Deny result of Besson \cite{Besson-BeurlingDenyPrinciple} to show in that case the domination of the heat-kernel and resolvent operators, in the sense of positivity preserving operators (in the same direction considered in this work).

\section{Heat Semi-Group Properties} \label{sec:semigroup}

\subsection{Notation and Terminology}

In this section, we restrict our discussion to the case when $\mu$ is a  probability measure. 
Recall that $P_t = \exp(t \Delta_{g,\mu})$, $t  > 0$, denotes the heat semi-group on $(M^n,g,\mu)$. It is known (see \cite[Theorem 3.3]{GrigoryanHeatKernelsOnWeightedManifolds} and the references therein) that $P_t$ admits an integral representation by means of a measurable (and in fact smooth) kernel $p_t : M \times M \rightarrow \Real_+$ called the heat-kernel:
\[
P_t f (y) = \int f(x) p_t(x,y) d\mu(x) .
\]
Formally, when $\Delta_{g,\mu}$ admits a discrete spectrum $\set{\lambda_k}$, we may write:
\begin{equation} \label{eq:pt}
 p_t(x,y) = \sum_{k=1}^\infty \exp(-\lambda_k t) \varphi_k(x) \varphi_k(y) ,
\end{equation}
where $\varphi_k$ denotes the ($L^2(\mu)$-normalized) eigenfunction associated to $\lambda_k$. To make this rigorous, one has to ensure the convergence of this infinite series in some sense. For instance, when $P_t$ is a Hilbert-Schmidt operator, i.e. $\norm{P_t}_{HS}^2 := \sum_{k=1}^\infty \norm{P_t u_k}_{L^2(\mu)}^2 < \infty$ for some (any) orthonormal basis $\set{u_k}$ of $L^2(\mu)$, then it is compact, implying the discreteness of the spectrum of $P_t$ and therefore of $\Delta_{g,\mu}$. Furthermore, since:
\[
\norm{p_t}^2_{L^2(\mu \otimes \mu)} = \norm{P_t}^2_{HS} = \sum_{k=1}^\infty \exp(-2 t \lambda_k) < \infty ,
\]
we see that the series (\ref{eq:pt}) converges in $L^2(\mu \otimes \mu)$, implying the existence of a heat-kernel $p_t \in L^2(\mu \otimes \mu)$.
Note that by the semi-group property:
\[
Z(t) = \sum_{k=1}^\infty \exp(-t \lambda_k) = \int p_t(x,x) d\mu(x) = \int \int p_{t/2}^2(x,y) d\mu(y) d\mu(x) =  \norm{P_{t/2}}^2_{HS} .
\]
Clearly, $P_t$ is trace-class ($Z(t) = tr(P_t) < \infty$) iff $P_{t/2}$ is Hilbert-Schmidt (and similarly iff $P_{t/q}$ is of Schatten class $q \in [1,\infty)$), so we will refer to these properties interchangeably. 

\subsection{Heat-Kernel Trace Estimates}

We will require the following fundamental dimension-free Harnack inequality due to Wang \cite[Lemma 2.1]{WangIntegrabilityForLogSob}, \cite{Wang-EquivalenceOfCDAndHarnack}. It states that the $\CD(\rho,\infty)$ condition ($\rho \in \Real$) is equivalent to the property that for any bounded continuous function $f \in C_b(M)$ and $p > 1$:
\begin{equation} \label{eq:Harnack}
\abs{P_t f}^p(x) \leq P_t(\abs{f}^p)(y) \exp\brac{ \frac{p d(x,y)^2}{4 (p-1) t} h(\rho,t) } \;\;\; \forall x,y \in M ,
\end{equation}
where:
\[
h(\rho,t) := \frac{2 \rho t}{\exp(2 \rho t) - 1} \;\;\; (\text{and }h(0,t) = 1) .
\]

The next elegant lemma is due to R\"{o}ckner--Wang \cite[Lemma 2.2]{RocknerWang-HarnackForGeneralizedMehlerSemigroups} (in fact, the result holds under much greater generality than the one considered here); for completeness, we sketch its proof. 

\begin{lem}[R\"{o}ckner--Wang]
Assume that for some $p \in [1,\infty)$ and positive measurable $\Phi: M \times M \rightarrow \Real_+$ we have for any $f \in C_b(M)$:
\begin{equation} \label{eq:Harnack-Phi}
\abs{P_t f}^p(x) \leq P_t(\abs{f}^p)(y) \Phi(x,y) \;\;\; \forall x,y \in M . 
\end{equation}
Then the heat-kernel $p_t$ satisfies:
\[
\norm{p_t(x,\cdot)}_{L^{p'}(\mu)} \leq \brac{\int \frac{d\mu(y)}{\Phi(x,y)}}^{-1/p} \;\;\; \forall x \in M ,
\]
where $p' := \frac{p}{p-1}$ denotes the conjugate exponent to $p$. 
\end{lem}
\begin{proof}[Proof Sketch]
Divide (\ref{eq:Harnack-Phi}) by $\Phi(x,y)$ and integrate with respect to $d\mu(y)$. Taking supremum over all $f \in C_b(M)$ with $\norm{f}_{L^p(\mu)} \leq 1$, the $L^{p'}(\mu)$ estimate on the density follows. We remark that the existence of the transition density $p_t$ in a very general setting in fact follows from (\ref{eq:Harnack-Phi}), as explained in \cite{RocknerWang-HarnackForGeneralizedMehlerSemigroups}. 
\end{proof}

Combining this with the Harnack estimate (\ref{eq:Harnack}), 
we obtain under $\CD(\rho,\infty)$ that the heat-kernel $p_t$ satisfies:
\begin{equation} \label{eq:kernel-estimate}
\norm{p_t(x,\cdot)}_{L^{p'}(\mu)} \leq \brac{\int \exp\brac{-\frac{p d(x,y)^2}{4 (p-1) t} h(\rho,t)} d\mu(y)}^{-1/p} \;\;\; \forall p > 1 \;\; \forall x \in M . 
\end{equation}
In particular, applying this with $p=2$, integrating with respect to $d\mu(x)$, and applying Jensen's inequality, we record:
\begin{prop} \label{prop:Z}
Let $(M^n,g,\mu)$ satisfy $\CD(\rho,\infty)$, $\rho \in \Real$. Then assuming either of the expressions on the right-hand side below is finite, $P_{t/2}$ is Hilbert-Schmidt and satisfies:
\begin{align*}
Z(t) = \norm{P_{t/2}}^2_{HS} = \norm{p_{t/2}}^2_{L^2(\mu \otimes \mu)} & \leq \int \brac{\int \exp\brac{-\frac{d(x,y)^2}{2} \frac{h(\rho,t/2)}{t/2} } d\mu(x)}^{-1} d\mu(y) \\
& \leq  \int \int \exp\brac{\frac{d(x,y)^2}{ 2} \frac{h(\rho,t/2)}{t/2}} d\mu(x) d\mu(y) .
\end{align*}
\end{prop}

The example of the Gaussian space (satisfying $\CD(1,\infty)$) shows that the application of Jensen's inequality in the last inequality above may be detrimentally wasteful - while the first bound above is finite for all $t > 0$, correctly verifying that the Ornstein-Uhlenbeck semi-group $P_t$ is Hilbert-Schmidt for all positive times, the second bound is infinite for $t \in [0,\ln(3)]$. In any case, employing (\ref{eq:Z-to-lambda}), we obtain from this the following estimate on the individual eigenvalues for $n$-dimensional Gaussian space:
\[
\lambda_k \geq \log(k) - c n ,
\]
for some numeric constant $c > 0$. Comparing this to the correct eigenvalue behaviour (\ref{eq:Gaussian-eigenvalues}), we see that this general method only yields a reasonable estimate for large $n$ and $k \gg \exp(cn)$. For more general weighted-manifolds, we have:

\begin{prop} \label{prop:Z-LS}
Let $(M^n,g,\mu)$ satisfy $\CD(\rho,\infty)$, $\rho \in \Real$, and assume that it satisfies a log-Sobolev inequality:
\[
\int f^2 \log (f^2) d\mu - \int f^2 d\mu \log \brac{\int f^2 d\mu} \leq \frac{2}{L} \int \abs{\nabla f}^2 d\mu \;\;\; \forall f \in C^1(M,g) ,
\]
for some constant $L > 0$. Then for any $t > 0$ such that:
\[
s := \frac{h(\rho,t/2)}{t/2} < \frac{L}{2} ,
\]
and for any $x_0 \in M$, we have:
\begin{equation} \label{eq:Z-estimate}
Z(t) \leq \exp\brac{ \frac{2s}{1 - 2s/L} \int_M d(x,x_0)^2 d\mu(x) } < \infty .
\end{equation}
\end{prop}
\begin{proof}
The proof is immediate from the second estimate of Proposition \ref{prop:Z}, the inequality $d(x,y)^2 \leq 2 \brac{d(x,x_0)^2 + d(y,x_0)^2}$, and the following known consequence of the log-Sobolev inequality (see \cite{AidaMasudaShigekawa},\cite[Formula (2.4)]{BobkovGotzeLogSobolev}):
\[
\int \exp( s f^2) d\mu \leq \exp\brac{ \frac{s}{1 - 2 s / L} \int f^2 d\mu } \;\;\; \forall 0 < s < \frac{L}{2} \;\;\; \forall f \in \F(M,g)  .
\]
Finally, the log-Sobolev inequality implies by the Herbst argument \cite[Proposition 5.4.1]{BGL-Book} that Lipschitz functions have sub-Gaussian tail-decay, ensuring that $\int d(x,x_0)^2 d\mu(x) < \infty$. 
\end{proof}

Note that the inevitable dimension-dependence of the estimate (\ref{eq:Z-estimate}) is hidden in the expression $\int_M d(x,x_0)^2 d\mu(x)$ (see e.g. \cite[Section 8.4]{EMilmanGeometricApproachPartI}). For instance, this expression is equal to $n$ in the Gaussian example examined above. 
Applying (\ref{eq:Z-to-lambda}), we immediately obtain the following eigenvalue estimate:
\begin{cor} \label{cor:eigen-estimates}
Under the same assumptions as in Proposition \ref{prop:Z-LS}, we have for all $k \geq 1$:
\[
\lambda_k \geq \sup \set{ \frac{1}{t} \brac{\log k - \frac{2s}{1 - 2s/L} \int_M d(x,x_0)^2 d\mu(x) } \; ;\; t > 0 ~,~  s = \frac{h(\rho,t/2)}{t/2}  < \frac{L}{2} } .
\]
The latter supremum is over a non-empty set whenever $L > 4 (-\rho)_+ $. 
\end{cor}

 We emphasize that the above eigenvalue estimate may be strictly weaker and is less general than the one obtained by Wang in \cite[Corollary 5.5]{Wang-EigenvalueEstimates} (cf. \cite{Wang-SurveyOnSemiGroupAndSpectrum}). In particular, Wang's results yield a meaningful estimate for \emph{any} $L > 0$ and $\rho \in \Real$. The reason is that Wang cannot afford to pass through the heat-kernel trace $Z(t)$, which may be infinite, as he assumes that the space satisfies a general super-Poincar\'e inequality, which may be strictly weaker than our log-Sobolev assumption. Consequently, Wang employs a more delicate method for controlling the eigenvalues directly \cite[Theorem 5.1]{Wang-EigenvalueEstimates}; our approach has the advantage of being slightly simpler, yielding tractable estimates, albeit far less general. 

\begin{rem} \label{rem:scenario2}
Note that the asymptotic formula (\ref{eq:WeylForGammaP}) for the distribution of eigenvalues of $(\Real^n,\abs{\cdot},c_p^n \exp(-\frac{1}{p}\sum_{i=1}^n \abs{x_i}^p) dx)$ for $p \in (1,2)$,
shows that the associated semi-group $P_t$ can be Hilbert-Schmidt for all $t > 0$, while no log-Sobolev inequality is satisfied (since otherwise, by the Herbst argument, the Lipschitz function $x_1$ would have sub-Gaussian tail decay). 
\end{rem}

\subsection{Heat-kernel higher-order integrability}

While originally defined as an operator acting on $L^2(\mu)$, it is well-known and easy to see  that $P_t$ extends to an operator acting on $L^p(\mu)$ for all $p \in [1,\infty]$, and that it is contracting there: $\norm{P_t}_{L^p(\mu) \rightarrow L^p(\mu)} \leq 1$. $P_t$ is called hyperbounded with parameters $L > 0$ and $B \geq 0$ if for some (every) $1 < p < \infty$ and every $t > 0$:
\begin{equation} \label{eq:hyper}
\norm{P_t f}_{L^{q(t)}(\mu)} \leq \exp(\beta(t)) \norm{f}_{L^p(\mu)} \;\;\; \forall f \in L^p(\mu) 
\end{equation}
where:
\[
\frac{q(t) - 1}{p-1} := \exp(2 t L)  \; , \; \beta(t) := B \brac{\frac{1}{p} - \frac{1}{q(t)}} .
\]
It is called hypercontractive if $B=0$. It was observed by Gross \cite{Gross-LSI} that  hypercontractivity (hyperboundedness) is equivalent to the following (defective) log-Sobolev inequality (cf. \cite[Theorem 5.2.3]{BGL-Book}):
\[
\int f^2 \log(f^2) d\mu  - \int f^2 d\mu \log \brac{\int f^2 d\mu}  \leq \frac{2}{L} \int \abs{\nabla f}^2 d\mu + B \int f^2 d\mu \;\;\; \forall f \in C^1(M) . 
\]

\begin{rem} \label{rem:hyperbdd}
It is known (see \cite[Theorem 5.2.5]{BGL-Book}) that in fact it is enough for (\ref{eq:hyper}) to hold for a single $t_0 > 0$: if (\ref{eq:hyper}) hold at time $t_0 > 0$ for $p=2$, $q(t_0) > 2$ and $\beta(t_0) \geq 0$, then $P_t$ is hyperbounded with parameters $L = \frac{q(t_0)-2}{2 q(t_0) t_0}$ and $B = \beta(t_0) / (t_0 L)$. 
\end{rem}

\begin{rem} \label{rem:Miclo}
Remarkably, it was recently shown by Miclo \cite{Miclo-Hyperboundedness} by a type of compactness argument that hyperboundedness always implies spectral-gap $\lambda_2 > \lambda_1 = 0$, which by an argument of Rothaus can be used to tighten the hyperboundedness into genuine hypercontractivity. However, although this implies that hyperboundedness and hypercontractivity are equivalent properties, there is no (and there cannot be any) quantitative relation between the former and the latter, see \cite{Miclo-Hyperboundedness}. 
\end{rem}

\begin{prop} \label{prop:dichotomy}
The following four properties are equivalent:
\begin{enumerate}
\item and (1 bis).
The heat semi-group is hyperbounded and Hilbert-Schmidt:\\
For some (any) $q \in (2,\infty)$, there exists $t_q > 0$ so that $\norm{P_{t_q}}_{L^2(\mu) \rightarrow L^q(\mu)} \leq \beta_q <\infty$  and $\norm{P_{t_q}}_{HS} \leq Z_2  <\infty$. 
\item and (2 bis).
Higher-order integrability of the heat-kernel:\\ 
For some (any) $q \in (2,\infty)$, there exists $s_q > 0$ so that $\norm{p_{s_q}}_{L^q(\mu \otimes \mu)} \leq Z_q < \infty$.
\end{enumerate} 
\end{prop}

\begin{proof} \hfill

\noindent
$(1\;bis) \Rightarrow (1)$ and $(2\;bis) \Rightarrow (2)$ are trivial. \\

\noindent
$(1) \Rightarrow (1\;bis)$.\\
This follows from Remark \ref{rem:hyperbdd} and the fact that $t \mapsto \norm{P_t}_{HS}$ is non-increasing. \\

\noindent
$(1) \Rightarrow (2)$ for fixed $q > 2$. \\
The spectrum is discrete since $P_{t_q}$ is Hilbert-Schmidt and hence compact (in fact, by Remark \ref{rem:Miclo}, hyperboundedness implies hypercontractivity, which is known \cite{WangSuperPoincareAndIsoperimetry} in the finite-dimensional setting to imply discreteness of spectrum). Adapting an argument from \cite[Chapter 5.3]{BGL-Book}, we have by (\ref{eq:pt}):
\[
\norm{p_s}_{L^q(\mu \otimes \mu)} \leq \sum_{k=1}^\infty \exp(-\lambda_k s) \norm{\varphi_k}_{L^q(\mu)}^2 .
\]
But by assumption (1):
\[
\exp(-\lambda_k t_q) \norm{\varphi_k}_{L^q(\mu)} = \norm{P_{t_q} \varphi_k}_{L^q(\mu)} \leq \beta_q \norm{\varphi_k}_{L^2(\mu)}  = \beta_q ,
\]
and hence:
\[
\norm{p_s}_{L^q(\mu \otimes \mu)} \leq \beta_q^2  \sum_{k=1}^\infty \exp(\lambda_k (2 t_q - s)) . 
\]
Setting $s_q = 4 t_q$, we obtain the asserted bound (2):
\[
\norm{p_{s_q}}_{L^q(\mu \otimes \mu)} \leq \beta_q^2 \norm{P_{t_q}}^2_{HS} \leq \beta_q^2 Z_2^2 < \infty .
\]

\noindent
$(2) \Rightarrow (1)$ for fixed $q > 2$. \\
By H\"{o}lder's inequality, we always have:
\[
\norm{P_t f}_{L^q(\mu)} \leq \norm{p_t}_{L^q(\mu \otimes \mu)} \norm{f}_{L^{q'}(\mu)}  \; , \; q' = \frac{q}{q-1} .
\]
When $q > 2$, setting $t_q=s_q$, we have by Jensen's inequality and assumption (2):
\[
\norm{P_{t_q} f}_{L^q(\mu)} \leq Z_q \norm{f}_{L^2(\mu)} .
\]
In addition, we trivially have:
\[
\norm{P_{t_q}}_{HS} = \norm{p_{t_q}}_{L^2(\mu \otimes \mu)} \leq \norm{p_{t_q}}_{L^q(\mu \otimes \mu)} \leq Z_q ,
\]
yielding the desired (1).\\ 

\noindent
$(2) \Rightarrow (2\;bis)$.\\
This follows by the chain of implications $(2) \Rightarrow (1) \Rightarrow (1 \; bis) \Rightarrow (2 \; bis)$, 
where the last implication follows since the value of $q > 2$ remained unaltered in the equivalence $(1) \Leftrightarrow (2)$. 
\end{proof}

Note that the proof of Proposition \ref{prop:dichotomy} is very general, and so it is not restricted to the finite-dimensional weighted-manifold setting. 
This proposition entails the following trichotomy for the heat-kernel, which we find interesting and surprising: 

\begin{enumerate}
\item For all $t > 0$, $p_t \notin L^2(\mu \otimes \mu)$. \\
(this happens if the heat semi-group $P_t$ is \textbf{never} Hilbert-Schmidt, or equivalently, Schatten class $q \in [1,\infty)$). 
\item For all $t > t_2 \geq 0$, $p_{t} \in L^2(\mu \otimes \mu)$, but for any $t > 0$ and $q > 2$, $p_t \notin L^q(\mu \otimes \mu)$. \\
(this happens if the heat semi-group $P_t$ is eventually Hilbert-Schmidt, but is \textbf{never} hyperbounded).
\item For \textbf{any} $q \in (2,\infty)$, for all $t > t_q \geq 0$,  the heat kernel $p_{t}$ is in $L^q(\mu \otimes \mu)$. \\
(this happens iff $P_t$ is eventually Hilbert-Schmidt and hyperbounded).
\end{enumerate}
By Proposition \ref{prop:Z}, the third scenario applies to any weighted-manifold satisfying $\CD(\rho,\infty)$, $\rho \in \Real$, in conjunction with a log-Sobolev inequality with constant $L > 4 (-\rho)_{+}$ (in that case, the fact that $p_{t} \in L^q(\mu \otimes \mu)$ could have been deduced directly from (\ref{eq:kernel-estimate}), but surprisingly, this would require a stronger log-Sobolev constant when $\rho < 0$). In particular, the third scenario applies to the spaces $(\Real^n,\abs{\cdot},\nu^n_p)$ with $p \in [2,\infty]$. By Remark \ref{rem:scenario2}, the second scenario applies to the latter spaces when $p \in (1,2)$. The first scenario applies to the case $p=1$, since when the spectrum is not discrete $P_t$ cannot be compact and in particular is not Hilbert-Schmidt. In the infinite-dimensional setting, the first scenario can also apply when $P_t$ is hypercontractive (such as for the infinite dimensional Gaussian measure, or Wiener space, which do not have discrete spectrum). But in the finite-dimensional setting, we did not find an example of an eventually hyperbounded $P_t$ which is not Hilbert-Schmidt; Proposition \ref{prop:Z} indicates that the generalized Ricci curvature lower bound $\rho$ should be very negative for this to be possible (as we suspect is the case).

\section{Concluding Remarks} \label{sec:conclude}

The reader may have noticed that the body of evidence in support of Conjecture \conjtwo is somewhat different than the one for Conjecture \ref{conj:4}. On the one hand, we have Caffarelli's Contraction Theorem, which confirms Conjecture \conjtwo for Euclidean $\CD(\rho,\infty)$ spaces, and whose analogue is still missing in the $\CD(\rho,n)$ context of Conjecture \ref{conj:4}. But on the other hand, various volumetric and spectral results which are known for $\CD(\rho,n)$ manifolds are missing for the $\CD(\rho,\infty)$ case. 

The most elementary example of the latter is the Bishop--Gromov volume comparison theorem (e.g. \cite{GHLBookEdition3}). This classical result for $\CD(\rho,n)$ manifolds has since been generalized to the $\CD(\rho,N)$ setting for $N \in [n,\infty)$ (for weighted manifolds by Qian \cite{QianWeightedVolumeThms} and for very general metric-measure spaces by Sturm \cite{SturmCD12} and Lott--Villani \cite{LottVillaniGeneralizedRicci}). However, to the best of our knowledge, there is no sharp version for $\CD(\rho,\infty)$ spaces, perhaps since the latter condition does not feel the topological dimension $n$ of the underlying space, and since the associated Gaussian model space $(\Real^n,\abs{\cdot},\gamma^n_\rho)$ is non-homogeneous (for some non-sharp versions, see Sturm \cite{SturmCD12}). However, observe that Conjecture \conjtwo would imply that for any $(\Real^n,g,\mu)$ satisfying $\CD(\rho,\infty)$ with $\rho > 0$ and $\mu(\Real^n) = 1$, we have:
\begin{equation} \label{eq:BG}
\exists x_0 \in \Real^n \;\;\; \forall r > 0 \;\;\;  \mu(B_g(x_0,r)) \geq \gamma_\rho^n(B_{\abs{\cdot}}(0,r)) ,
\end{equation}
where $B_g(x_0,r)$ denotes the geodesic ball of radius $r$ about $x_0$ with respect to the metric $g$. 
Indeed, if $T : (\Real^n,\abs{\cdot},\gamma_\rho^n) \rightarrow (\Real^n,g,\mu)$ is the contracting map pushing forward $\gamma_\rho^n$ onto $\mu$, simply set $x_0 = T(0)$ and use the contraction property. In particular, this would imply that:
\begin{equation} \label{eq:BG2}
\exists x_0 \in \Real^n \;\;\; \int d(x,x_0)^2 d\mu(x) \leq \int \abs{x}^2 d\gamma_\rho^n(x) = \frac{n}{\rho} ,
\end{equation}
which for example would be useful to know for obtaining eigenvalue estimates in Corollary \ref{cor:eigen-estimates}.
We stress that we do not know how to prove even this simple consequence of Conjecture \conjtwo, which is somewhat of an ominous sign in regards to the latter. Of course, by Caffarelli's Contraction Theorem, (\ref{eq:BG}) and (\ref{eq:BG2}) are true when the underlying space is Euclidean. 

Consequently, we are somewhat skeptical of the validity of the tentative Conjectures \conjone and \conjtwo, which we mainly stated for motivational purposes. A much safer variant of these tentative conjectures, which we refrained from stating explicitly in the Introduction for the sake of expositional simplicity and coherence, is to restrict our attention to weighted manifolds satisfying a Graded Curvature-Dimension condition, to be introduced in \cite{EMilman-GradedCD}. With this more restrictive condition, instead of requiring a single lower bound on $\Ric_{g,\mu} = \Ric_g + \nabla^2_g V$, one imposes separate bounds on the two components of the latter tensor  (or some other linear combination thereof). Specifically we make the following:
 
\begin{conj12}
The tentative Conjectures \conjone and \conjtwo hold true when restricted to weighted manifolds satisfying $\Ric_{g} \geq 0$ and $\nabla^2_g V \geq \rho g$ (and still diffeomorphic to $\Real^n$).
\end{conj12}
Note that in any case, since $\rho > 0$, a manifold as above cannot be compact, since otherwise at the point where the maximum of $V$ is attained we would have $\nabla^2_g V \leq 0$. For such weighted manifolds, we can in fact prove (\ref{eq:BG}) using the usual Jacobi field approach - see the forthcoming \cite{EMilman-GradedCD}.

\setlinespacing{0.82}
\setlength{\bibspacing}{2pt}

\bibliographystyle{plain}
\bibliography{../../../ConvexBib}

\end{document}